\RequirePackage{fix-cm}
\documentclass[smallextended]{svjour3}       % onecolumn (second format)
\smartqed  % flush right qed marks, e.g. at end of proof
\usepackage{graphicx}
\usepackage{algorithm}
\usepackage{algorithmic}
\usepackage{amsmath,amsfonts}

\newcommand{\ba}{\mbox{\boldmath $a$}}
\newcommand{\bb}{\mbox{\boldmath $b$}}

\newcommand{\bd}{\mbox{\boldmath $d$}}

\newcommand{\bu}{\mbox{\boldmath $u$}}
\newcommand{\bv}{\mbox{\boldmath $v$}}

\newcommand{\bx}{\mbox{\boldmath $x$}}
\newcommand{\by}{\mbox{\boldmath $y$}}
\newcommand{\bz}{\mbox{\boldmath $z$}}

\newcommand{\bA}{\mbox{\boldmath $A$}}
\newcommand{\bB}{\mbox{\boldmath $B$}}

\newcommand{\XX}{{\cal X}}

\newcommand{\CC}{{\cal C}}

\newcommand{\GG}{{\cal G}}

\newcommand{\RR}{\mathbb R}

\newcommand{\Th}{{\rm Th}}
\journalname{}
\begin{document}

\title{Linearly-Convergent FISTA Variant for Composite Optimization with Duality %\thanks{Grants or other notes
%about the article that should go on the front page should be
%placed here. General acknowledgments should be placed at the end of the article.}
}
%\subtitle{Do you have a subtitle?\\ If so, write it here}

\titlerunning{Linearly-Convergent FISTA Variant for Composite Optimization with Duality}        % if too long for running head

\author{$\text{Casey Garner}^{1}$  \and $\text{Shuzhong Zhang}^2$
         %etc.
}

\authorrunning{Casey Garner \and Shuzhong Zhang} % if too long for running head

\institute{
%{\bf1.} Corresponding Author
%\and 
{\bf1.} Department of Mathematics, University of Minnesota \at
              Church Street SE, Minneapolis, MN, 55414, USA \\
               {\it Corresponding Author E-mail:} {garne214@umn.edu}
%              Tel.: +123-45-678910\\
%              Fax: +123-45-678910\\
                         %  \\
%             \emph{Present address:} of F. Author  %  if needed
           \and
           {\bf2.} Department of Industrial and Systems Engineering, University of Minnesota \at
           Union Street SE, Minneapolis, MN, 55414, USA \\ 
           \email{zhangs@umn.edu}
}

\date{
$\bullet$ This preprint has not undergone peer review or any post-submission improvements or corrections. The Version of Record of this
article is published in {\it Journal of Scientific Computing}, and is available online at \url{https://doi.org/10.1007/s10915-023-02101-z}.
%Received: date / Accepted: date
}
% The correct dates will be entered by the editor

\maketitle

\begin{abstract}
Many large-scale optimization problems can be expressed as composite optimization models. Accelerated first-order methods such as the fast iterative shrinkage-thresholding algorithm (FISTA) have proven effective for numerous large composite models. In this paper, we present a new variation of FISTA, to be called C-FISTA, which obtains global linear convergence for a broader class of composite models than many of the latest FISTA variants. We demonstrate the versatility and effectiveness of C-FISTA through multiple numerical experiments on group Lasso, group logistic regression and geometric programming models. Furthermore, we utilize Fenchel duality to show C-FISTA can solve the dual of a finite sum convex optimization model.
% \PACS{PACS code1 \and PACS code2 \and more}
\end{abstract}

%============================================================
\section{Introduction}
%============================================================
The following composite optimization model,
\begin{equation}\label{eqn:genCCO}
\begin{array}{lll}
& \min & {H}(\mathbf{B}(\bx)) + {R}(\bx) \\
& \text{s.t.} & \bx \in \mathcal{X}\subseteq \mathbb{R}^n,
\end{array}
\end{equation}
where $\mathcal{X}$ is a closed, convex subset of $\RR^n$, ${H}: \RR^m \rightarrow \RR$ is a smooth, convex function,  $R: \RR^n \rightarrow \RR$ is convex but potentially non-smooth, $\mathbf{B}:\RR^n \rightarrow \RR^m$ is a smooth mapping, and $H\circ \bB$ is a convex function over $\mathcal{X}$, has received ample attention in the literature \cite{BT09,BF95,DDL18,DDP19,N13,W90}. A special case of \eqref{eqn:genCCO} of immense importance is the additive composite optimization model,
\begin{equation}\label{eqn:ACCO}
\begin{array}{lll}
& \min & H(\bx) + R(\bx) \\
& \mbox{s.t.} & \bx \in \XX\subseteq \RR^n,
\end{array}
\end{equation}
which envelopes a plethora of models including: compressive sensing \cite{CP16}, Lasso \cite{T96}, and group Lasso and group logistic regression models \cite{QSG13,YLY11} as well as several machine learning constructs such as support vector machines. Due to their dimensionality, many large-scale optimization models have rendered second-order methods computationally impractical; thus, efficient and accelerated first-order algorithms have become essential for tackling numerous problems. With the advent of Nesterov's seminal work \cite{N83} much effort has been expended toward the development of accelerated first-order methods. For a comprehensive overview of this line of research see the monograph of d'Aspremont {\it et al.}\/ and the references therein \cite{AST21}. As an example, utilizing Nesterov acceleration, Beck and Teboulle developed the influential fast iterative shrinkage-thresholding algorithm (FISTA) \cite{BT09} which obtained the optimal sublinear convergence rate $\mathcal{O}(1/k^2)$ for \eqref{eqn:ACCO} with $H$ and $R$ convex and $R$ possibly non smooth. In recent years, under the assumption $H$ and potentially $R$ are strongly convex, many accelerated versions of FISTA for \eqref{eqn:ACCO} have been developed which obtain the optimal linear convergence rate proved by Nesterov \cite{CC19,CP16,FV19,FV20,RC21}.

Furthermore, recent work has been done to determine conditions which lighten the assumption of strong convexity while maintaining accelerated and at times linear convergence \cite{BS17,DDL18,ING19}. In \cite{DDL18} the authors demonstrate local linear convergence to a first-order stationary point for \eqref{eqn:genCCO} without assuming any strong convexity on $H$. The authors extended the work in \cite{ZZ17} by utilizing a specific error bound condition (Definition 3.1 in \cite{DDL18}) in tandem with an assumption on the quadratic growth of the objective function.

In this paper we develop an accelerated composite version of FISTA, {C-FISTA}, similar to the recent FISTA variants which extends to solving model \eqref{eqn:genCCO} as well as model \eqref{eqn:ACCO}. Also, in line with the recent works \cite{DDL18,DDP19}, we forgo any assumption guaranteeing strong convexity of the objective function in \eqref{eqn:genCCO} and prove {C-FISTA} obtains global linear convergence under certain error bound conditions.
%============================================================
\subsection{Main Contributions}
This paper presents three main contributions. The first contribution of this paper is the development of an accelerated version of FISTA which obtains global linear convergence for \eqref{eqn:genCCO} under specific assumptions. Our algorithm, {C-FISTA}, differs from the aforementioned FISTA variants because they were either exclusively designed for the additive composite model \eqref{eqn:ACCO} or only obtained local instead of global linear convergence for the more generalized model \eqref{eqn:genCCO}. We demonstrate that {C-FISTA} is a generalization of the FISTA variant GFISTA \cite{CC19,CP16} when $\bB(\bx) = \bx$ proving {C-FISTA} is an extension of recent algorithmic developments. Further, we present an alternative convergence analysis than those presented for the other accelerated FISTA variants and exclude any reference to the strong convexity parameter of $R$ as done in \cite{CC19,FV19,FV20}.

The second contribution of this paper is our utilization of Fenchel duality to develop dual models which can be solved via {C-FISTA}, and our leveraging of Fenchel duality theory \cite{HL88,RF70,Roos20} to enable efficient computation of the subproblems of {C-FISTA}. Additionally, we outline a dual algorithmic approach for a general convex model which can be solved by {C-FISTA} and generalizes the approaches of Han and Lou \cite{HL88} and Auslender \cite{A92} among others \cite{FHN96,GM89}. Sections \ref{Fenchel Duality} and \ref{duality-subproblems} provide the relevant background on Fenchel duality and provide an example on how we utilize the theory in the implementation of {C-FISTA}.

The third contribution of this paper is the presentation of globally linearly convergent algorithms for solving various Lasso and logistic regression models with {C-FISTA}. In our numerical experiments, we demonstrate {C-FISTA} outperforms FISTA \cite{BT09} and the SLEP software package \cite{SLEP} for solving group, sparse-group and overlapping sparse-group Lasso models and sparse-group logistic regression models. We also compare {C-FISTA} with ADMM on the Lasso models, which has proven linear convergence for various Lasso formulations \cite{HQL17}, and demonstrate superior convergence after tuning ADMM in our numerical experiments. Lastly, we demonstrate C-FISTA's applicability for solving a class of geometric programming models.

%============================================================
\subsection{Organization of the Paper}
In Section \ref{C-FISTA} we present {C-FISTA} and prove the global linear convergence of the algorithm under certain assumptions. Section \ref{Lasso and log reg models} motivates {C-FISTA} by describing the Lasso, logistic regression and geometric programming models which are solvable with {C-FISTA}. In Section \ref{Fenchel Duality} we present the necessary background on Fenchel duality and make connections to our accelerated FISTA algorithm, and Section \ref{duality-subproblems}  provides an example of how Fenchel duality informs our application of {C-FISTA} to solve the crucial proximal mapping step. Section \ref{numerical results} contains the algorithms for solving the Lasso, sparse-group logistic regression and geometric programming models with {C-FISTA}, and our numerical experiments comparing {C-FISTA} to other efficient first-order algorithms. The paper concludes in Section \ref{conclusions} with some final remarks and potential avenues for future investigation.

%============================================================================
\section{A Generalized Composite Optimization Algorithm}\label{C-FISTA}
In this paper, our primary focus revolves around the composite optimization model,
\begin{equation} \label{genSCOPT}
\begin{array}{lll}
    & \min & H( \bB(\bx)) + R(\bx) \\
    & \mbox{s.t.} & \bx \in \XX \subseteq \RR^n, \nonumber
\end{array}
\end{equation}
where $\mathcal{X}$ is a closed, convex subset of $\RR^n$, $H : \RR^m \rightarrow \RR$ is a smooth, convex function,
%strongly convex with parameter $\mu>0$ and has Lipschitz gradient with parameter $L > 0$,  %{\color{blue} and $H$ is componentwise nondecreasing},
$R : \RR^n \rightarrow \RR$ is convex but potentially non-smooth, $\bB: \RR^n \rightarrow \RR^m$ is smooth and defined as $\bB(\bx) = \left( B_1(\bx), \cdots, B_m(\bx)\right)^\top$ with Jacobian, $\mathbf{J_B}(\bx)$, at $\bx\in \RR^n$. For convenience in our analysis we denote $F(\bx):=H(\bB(\bx))+R(\bx)$. A key assumption for feasible implementation is that the following proximal mapping with respect to $R$ is efficiently computable:
\[
\mbox{\rm Prox}_{ t R}(\bu):=
\mbox{arg}\min_{\bx\in \XX}\left\{ \frac{1}{2} \| \bx - \bu\|^2 + t R(\bx)\right\},
\]
where $t>0$. We now state the four main assumptions in our analysis: 
\vspace{0.1in}
\begin{itemize}
\item[{\bf(A0)}] $H$ is strongly convex with parameter $\mu$ and gradient Lipschitz with parameter $L$, and $H\circ \bB$ is convex over $\XX$.
\vspace{1mm}
\item[{\bf(A1)}] $B_i \in C^2(\XX)$ is Lipschitz continuous with parameter $\bar L_i \geq 0$, and has Lipschitz gradient constant $L_i \geq 0 $ for $i=1,\hdots, m$.
\vspace{1mm}
\item[{\bf (A2)}] There exists $\tau > 0$ such that $\tau \| \bx- \by \|^2 \leq \|\bB(\bx) - \bB(\by)\|^2$ for all $\bx,\by \in \text{span}(\mathcal{X})$.
\vspace{1mm}
\item[\bf{(A3)}] There exists $\xi \ge  0$ such that  $\forall \; \bx, \by \in \text{span}(\mathcal{X})$ with $\bar{\by} := \bB(\by)$,
\[ \bigg|\nabla_{\bar{\by}} H(\bar{\by})^\top \left( \mathbf{J_B}(\bx) - \mathbf{J_B}(\by) \right)(\bx - \by)\bigg| \leq \xi \| \bx - \by\|^2.\]
\end{itemize}

Our analysis focuses on the proposed algorithm C-FISTA (Algorithm \ref{alg:C FISTA alg}). However, before stating the algorithm and proving our convergence results, we further frame the assumptions. 

\begin{algorithm}
\caption{C-FISTA}
\begin{algorithmic}\label{alg:C FISTA alg}
\STATE \textbf{Input:} Constants $r, L, \mu, \xi$ and $\tau$. (See assumptions (A2), (A3), and \eqref{upper-bound})
\STATE \textbf{Step 0.} Choose any $(\bx^0,\bz^0) \in \XX\times \XX$. Let $k:=0$
\STATE \textbf{Step 1.} Let
\begin{equation} \label{y-def2}
\by^k :=\frac{1}{1+\theta} \, \bx^k + \frac{\theta}{1+\theta} \, \bz^k,
\end{equation}
and 
\begin{eqnarray}\label{eqn:x-def}
\hspace{0.2in} \bx^{k+1} &:=& \mbox{arg}\min_{\bx\in \XX}  H(\bar \by^k) + \nabla_{\by} H(\bar \by^k)^\top (\bx-\by^k) + \frac{rL}{2} \| \bx - \by^k\|^2 + R(\bx), \label{x-def2} \\
&=& \mbox{\rm Prox}_{ R / (rL) }\left(\by^k-\frac{1}{rL} \nabla_{\by} H(\bar \by^k)\right), \nonumber
\end{eqnarray}
and
\begin{equation} \label{z-def2}
\bz^{k+1}:= (1-\theta) \bz^k + \theta \by^k + \alpha (\bx^{k+1}- \by^k),
\end{equation}
%\alpha:=\sqrt{\frac{r L - \xi}{\tau \mu - \xi}},\,\, \theta := \frac{\sqrt{(\tau\mu - \xi)(rL -\xi)}}{rL}
where from \eqref{eqn:thetaANDalpha}, 
\[ \theta := \frac{\sqrt{(\tau\mu - \xi)(rL -\xi)}}{rL}\;\; \text{ and } \;\; \alpha:=\sqrt{\frac{r L - \xi}{\tau \mu - \xi}}.\]%\vspace{3mm}
\STATE \textbf{Step 2.} Let $k:=k+1$ and return to {\bf Step 1} until convergence.
\end{algorithmic}
\end{algorithm}

%===========================================================
\subsection{ Discussion of Assumptions}\label{sec:assumptions}

In this section we further detail the assumptions and demonstrate how they compare to standard assumptions in the literature. Before expounding upon (A0) - (A3), we first make the essential observation that under these assumptions it is possible $H(\bB(\bx))$ is not strongly convex in $\bx$. Letting $\bar{\bx}=\bB(\bx)$, we do see, however, that the descent inequality and the strongly convex inequality still hold for $H$ with respect to $\bar{\bx}$ and $\bar{\by}$, i.e.\ for all $\bar{\bx}, \bar{\by} \in \RR^m$,

\vspace{-0.1in}
\begin{eqnarray}
H(\bar{\by}) \le H(\bar \bx) + \nabla_{\bar{\bx}} H(\bar{\bx})^\top (\bar{\by} - \bar{\bx}) + \frac{L}{2} \| \bar{\by} - \bar{\bx}\|^2, \nonumber \\
H(\bar{\by}) \ge H(\bar{\bx}) + \nabla_{\bar{\bx}} H(\bar{\bx})^\top (\bar{\by} - \bar{\bx}) + \frac{\mu}{2} \| \bar{\by} - \bar{\bx}\|^2, \nonumber
\end{eqnarray}
where $\nabla_{\bar{\bx}} H(\bar{\bx}) := \nabla H (\by) \big|_{\by =\bB(\bx)}$ and
\[\nabla_{\bx} H(\bar{\bx}) = \nabla_{\bx} H(\bB(\bx)) = \mathbf{J_B}(\bx)^\top \nabla_{\bar{\bx}} H(\bar{\bx}). \]
For the remainder of the paper, we will reserve $\bar{\bx}$ to denote $\bB(\bx)$, that is $\bar{\bx}:=\bB(\bx)$.

Of the assumptions, (A0) and (A1), are the most straightforward. The first assumption is standard in the literature, and the second ensures $\bB$ is sufficiently well-behaved in terms of its continuity and differentiability; however, no convexity assumptions are made directly on the component functions $B_i$. 

\begin{remark}
From assumption (A1) we obtain a few important inequalities. First, utilizing the fact each $B_i$ is Lipschitz continuous, it follows from the limit definition of the directional derivative that $\| \nabla B_i (\bx) \|^2 \leq \bar L_i^2$ for all $\bx\in \RR^n$. Therefore, letting $\bx\in \RR^n$ be fixed we see that,
\begin{align}
\|\mathbf{J_B}(\bx)\| &= \sup_{\mathbf{v} \in \RR^n} \left\{ \|\mathbf{J_B}(\bx) \bv \| \; \bigg| \; \|\bv\| =1  \right\} \nonumber \\
			&= \sup_{\bv \in \RR^n} \left\{ \sqrt{ \sum_{i=1}^{m} \left( \nabla B_i(\bx)^\top \bv \right)^2 } \; \bigg| \; \|\bv\| =1  \right\} \nonumber \\
			&\leq \sqrt{ \bar L_1^2 + \hdots + \bar L_m^2 }. \nonumber
\end{align}
So, defining,
\begin{equation}\label{eqn:r}
r:=  \bar L_1^2 + \cdots + \bar L_m^2,
\end{equation}
we have,
\begin{equation}\label{eqn:jacob_bound}
\|\mathbf{J_B}(\bx)\|^2 \leq r, \;\; \forall \bx \in \RR^n.
\end{equation}

As a result of \eqref{eqn:jacob_bound} we obtain a second important inequality for our analysis. Letting $\bx,\by \in\RR^n$, by the differentiability of $\bB$,
\[
\bB(\by) - \bB(\bx)  = \int_{0}^{1} \mathbf{J_B}\left( \bx + t(\by-\bx) \right)(\by-\bx)dt,
\]
\vspace{-0.1in}
which by the triangle inequality and \eqref{eqn:jacob_bound} implies,
\[
\|\bB(\by) - \bB(\bx)\| \leq \int_{0}^{1} \| \mathbf{J_B}\left( \bx + t(\by-\bx) \right)\|\cdot \|\by-\bx\| dt \leq \sqrt{r} \| \by -\bx \|,\]
and so,
\
\begin{equation}\label{upper-bound}
\| \bar \by  - \bar \bx \|^2  =  \|\bB(\by) - \bB(\bx) \|^2 \leq r \|\by - \bx \|^2, \;\forall\; \bx ,\by \in \RR^n.
\end{equation}
\end{remark}

In regards to (A2), we note it is a condition which depends as much on the constraint set $\XX$ as it does on the function $\bB$. For example, if $\bB(\bx) = \bA\bx-\bb$ with $\bA \in \mathbb{R}^{m \times n}$ singular, then  (A2) will fail to hold if $\text{span}{(\XX)} = \mathbb{R}^n$; however, it will be satisfied if $\text{span}{(\XX)} \cap \text{null}(\bA) = \{\mathbf{0}\}$. Additionally, this condition is fundamentally different than growth conditions about the set of optimal solutions such as given in \cite{DDL18} because it places no condition on the proximal mapping nor relates to the set of optimal solutions of \eqref{eqn:genCCO}. 

Assumption (A3) by all appearances is the most opaque; however, (A3) is essentially equivalent to assuming $H\circ \bB$ is gradient Lipschitz continuous. If it is assumed $H\circ \bB$ is gradient Lipschitz with parameter $\hat{L} >0$, then (A3) follows with $\xi = \hat{L} + rL$. Similarly,  if assumption (A3) holds, then $H \circ \bB$ is gradient Lipschitz with parameter $\xi + rL$. Therefore, this assumption is fundamentally a constraint on the gradient of $H\circ \bB$ which is unlike standard error bound or growth conditions found in the literature.

As a closing remark on the assumptions, which we will revisit following the proof of Theorem \ref{thm:iter-complexity}, although (A2) and (A3) are stated in a global fashion they only need to hold near the iterates generated by Algorithm \ref{alg:C FISTA alg}. For the sake of our argument, we require our assumptions to hold on $\text{span}(\XX)$; however, in practice this is unnecessary and we demonstrate this through several numerical experiments in Section \ref{numerical results}. 

\subsubsection{Example Models} Before beginning our analysis of Algorithm \ref{alg:C FISTA alg} it makes sense to introduce a few models which will appear frequently in the applications presented in the forthcoming sections, and provide an exposition on their relationship to the stated assumptions.
\vspace{0.1in}

\begin{itemize}
\item 
If  $\bB(\bx)$ is affine linear, i.e. $\bB(\bx) = \bA\bx - \bb$ for $\bA \in \RR^{ m \times n }$ and $\bb\in \RR^m$, then (A1) is met and the convexity of $H\circ \bB$ follows readily without the need for any additional assumptions on $H$. Furthermore, (A3) is trivially satisfied with $\xi = 0$ since $\mathbf{J_B}(\bx) = \bA$. As previously stated, (A2) will hold provided $ \text{span}(\mathcal{X}) \cap \text{null}(\bA) = \{\mathbf{0}\}$. Finally, from \eqref{eqn:r} we see $r$ in \eqref{upper-bound} can by given as the largest eigenvalue of $\bA^\top \bA$. Thus, in this setting, we always have $ r \leq \lambda_{\text{max}}(\bA^\top \bA)$ and $\xi = 0$. \\

\item In the geometric programming model discussed in Section \ref{GeoProg}, we have $\bB(\bx) = \left(\ln(x_1), \hdots, \ln(x_n) \right)^\top$ with $\XX$ a closed, bounded, convex subset of the positive orthant in $\RR^n$. Thus, (A1) holds  and, as will be shown in Section \ref{GeoProg}, constants $r$ and $\tau$ can be computed which satisfy the necessary inequalities over the constraint set. This example is significant because it demonstrates non-linear and non-convex choices for $\bB$ are admissible for certain models.
\end{itemize}

\vspace{-0.2in}
\subsection{C-FISTA Convergence Analysis} 

We now state the formal description of the C-FISTA algorithm, and present the main convergence result which provides a global linear convergence guarantee. 

\begin{theorem} \label{thm:iter-complexity}
Assuming assumptions (A1) - (A3) hold with $ \tau \mu  - \xi> 0$ and $H \circ \bB$ convex, then {C-FISTA} has an accelerated global linear rate of convergence for \eqref{eqn:genCCO}; that is, for each iteration, we have,
\begin{multline}
F(\bx^{k+1}) - F(\bx^*) + \left(\frac{\tau \mu - \xi}{2}\right) \|\bz^{k+1} - \bx^*\|^2 \le \nonumber \\ \left(1- \frac{\sqrt{(\tau\mu - \xi)(rL -\xi)}}{rL} \right) \left( F(\bx^{k}) - F(\bx^*) +\left( \frac{\tau \mu - \xi}{2}\right) \|\bz^{k} - \bx^*\|^2 \right).
\end{multline}
In particular, if we perform C-FISTA for $k$ iterations, then it holds that,
\begin{multline}
F(\bx^{k}) - F(\bx^*) \le \nonumber \\ \left( 1 -\frac{\sqrt{(\tau\mu - \xi)(rL -\xi)}}{rL}\right)^k \left( F(\bx^{0}) - F(\bx^*) + \left(\frac{\tau \mu - \xi}{2}\right) \|\bz^{0} - \bx^*\|^2 \right). \nonumber 
\end{multline}
\end{theorem}

\begin{proof}
By the Lipschitz inequality for $H$,
\begin{eqnarray}
&H&(\bar \bx^{k+1}) \nonumber \\
& \le & H(\bar \by^{k}) + \nabla_{\bar \by} H(\bar \by^k)^\top \left( \bar \bx^{k+1} - \bar \by^k \right) + \frac{L}{2} \| \bar \bx^{k+1} - \bar \by^{k} \|^2,  \nonumber \\
& \le & H(\bar \by^{k}) + \nabla_{\bar \by} H(\bar \by^k)^\top \left( \bB( \bx^{k+1}) - \bB(\by^k) \right) + \frac{rL}{2} \| \bx^{k+1} -  \by^{k} \|^2,  \nonumber \\
& = & H(\bar \by^{k}) + \nabla_{\bar \by} H(\bar \by^k)^\top \int_{0}^{1} \mathbf{J_B}\left(\by^k + t\left(\bx^{k+1} - \by^k\right)\right)\left( \bx^{k+1} - \by^k \right)dt \nonumber \\
&&+ \frac{rL}{2} \| \bx^{k+1} -  \by^{k} \|^2,  \nonumber \\
& = & H(\bar \by^{k}) + \nabla_{\bar \by} H(\bar \by^k)^\top \hspace{-0.07in} \int_{0}^{1} \hspace{-0.07in}  \left[ \mathbf{J_B}\left(\by^k + t\left(\bx^{k+1} - \by^k\right)\right) - \mathbf{J_B}\left( \by^k\right) \right] \left( \bx^{k+1} - \by^k \right)dt  \nonumber  \\ & &+ \nabla_{\bar \by} H(\bar \by^k)^\top \mathbf{J_B}\left( \by^k\right) \left( \bx^{k+1} - \by^k \right) +   \frac{rL}{2} \| \bx^{k+1} -  \by^{k} \|^2,  \nonumber
\end{eqnarray}
where the second line follows from \eqref{upper-bound} and the last lines are a result of the Newton-Leibniz formula.  Applying assumption (A3) we obtain,
\begin{eqnarray}
&H&(\bar \bx^{k+1}) \nonumber \\
&\leq & H(\bar \by^{k}) + \hspace{-0.07in} \int_{0}^{1} \hspace{-0.07in} \xi t \| \bx^{k+1} - \by^k \|^2 dt  \nonumber + \nabla_{\by} H(\bar \by^k)^\top \left( \bx^{k+1} - \by^k \right) +   \frac{rL}{2} \| \bx^{k+1} -  \by^{k} \|^2,  \nonumber \\
& = & H(\bar \by^{k}) + \nabla_{\by} H(\bar \by^k)^\top \left( \bx^{k+1} - \by^k \right) +   \frac{rL + \xi}{2} \| \bx^{k+1} -  \by^{k} \|^2. \nonumber
\end{eqnarray}
Utilizing the strong convexity of $H$, for all $x\in \XX$ we have,
\begin{eqnarray}%\label{H_bound}
H(\bar \bx^{k+1}) & \le & H(\bar \bx)  - \nabla_{\bar \by} H( \bar \by^k )^\top \left( \bar{\bx} - \bar{\by}^k \right) - \frac{\mu}{2} \| \bar \bx - \bar \by^k \|^2 \nonumber \\
&& \hspace{0.75in} + \nabla_{\by} H(\bar \by^k)^\top \left( \bx^{k+1} - \by^k \right) +  \frac{rL + \xi}{2} \| \bx^{k+1} -  \by^{k}\|^2, \nonumber \\
& \le & H(\bar \bx)  - \nabla_{\bar \by} H( \bar \by^k )^\top \left( \bar{\bx} - \bar{\by}^k \right) + \nabla_{\by} H(\bar \by^k)^\top \left( \bx^{k+1} - \by^k \right)  \nonumber \\
&&\hspace{0.75in}+   \frac{rL + \xi}{2} \| \bx^{k+1} -  \by^{k}\|^2 - \frac{\tau \mu}{2} \| \bx - \by^k \|^2, \nonumber
\end{eqnarray}
where the second inequality follows from (A2). Applying (A3) and utilizing a similar argument with the Newton-Leibniz formula it follows, 
\begin{eqnarray}\label{H_bound}
H(\bar \bx^{k+1}) & \le & H(\bar \bx)  + \frac{\xi}{2}\| \bx - \by^k \|^2 - \nabla_{\by} H(\bar \by^k)^\top \left( \bx - \by^k \right) \nonumber \\ & &  \hspace{0.05in} + \nabla_{\by} H(\bar \by^k)^\top \left( \bx^{k+1} - \by^k \right)  +   \frac{rL + \xi}{2} \| \bx^{k+1} -  \by^{k}\|^2 - \frac{\tau \mu}{2} \| \bx - \by^k \|^2 \nonumber \\
 &\le& H(\bar \bx)  + \nabla_{\by} H(\bar \by^k)^\top \left( \bx^{k+1} - \bx \right) \nonumber \\
 && \hspace{0.05in} +   \frac{rL + \xi}{2} \| \bx^{k+1} -  \by^{k}\|^2 - \frac{\tau \mu - \xi}{2} \| \bx - \by^k \|^2.
\end{eqnarray}
By the first order optimality conditions of \eqref{x-def2} we have,
\[
\left( \nabla_{\by} H(\bar \by^k) + r L \left( \bx^{k+1} - \by^k \right) + R'(\bx^{k+1}) \right)^\top \left( \bx - \bx^{k+1} \right) \ge 0,\,\, \mbox{ for all $\bx\in \XX$},
\]
where $R'(\bx^{k+1})$ is an element of the subgradient of $R$ at $\bx^{k+1}$. The optimality conditions above along with the convexity of $R$ and \eqref{H_bound} give us,
\begin{eqnarray}\label{obj-value1}
F(\bx^{k+1})& \le & F(\bx) + r L (\bx^{k+1}-\by^k)^\top (\bx-\bx^{k+1}) + \frac{rL + \xi}{2} \| \bx^{k+1} -  \by^{k}\|^2 \nonumber \\ & & \hspace{1.0in} - \frac{\tau \mu - \xi}{2} \| \bx - \by^k \|^2 ,\, \forall \bx \in \XX.
\end{eqnarray}
Let $\theta\in (0,1)$, whose exact value is to be determined later. Take $\bx=\bx^k$ in \eqref{obj-value1} and multiply by $1-\theta$ on both sides of the expression. Then, let $\bx=\bx^*$ in \eqref{obj-value1} and multiply by $\theta$ on both sides. Finally, adding up these two resultant inequalities and applying the assumption $\tau \mu - \xi > 0$ we have,
\begin{eqnarray} \label{CP-b}
&F&(\bx^{k+1}) - F(\bx^*) \nonumber \\
&\le&  (1-\theta) \left( F(\bx^{k}) - F(\bx^*) \right) - \theta \left(\frac{\tau\mu - \xi}{2}\right)\|\bx^*-\by^k\|^2 \nonumber \\
&+& r L (\bx^{k+1}-\by^k)^\top \hspace{-0.05in} \left[(1-\theta) \bx^k + \theta \bx^* - \bx^{k+1} \right] + \left(\frac{rL + \xi}{2}\right) \|\bx^{k+1}-\by^k\|^2. \hspace{0.3in}
\end{eqnarray}

Recall from \eqref{z-def2} that,
\[
\bz^{k+1}=(1-\theta) \bz^k + \theta \by^k + \alpha (\bx^{k+1}-\by^k),
\]
where the value of the parameter $\alpha$ will be determined later. Thus,
\begin{eqnarray}
\|\bz^{k+1} - \bx^*\|^2 &\le & (1-\theta) \| \bz^k-\bx^* \|^2 + \theta \| \by^k - \bx^*\|^2 + \alpha^2 \|\bx^{k+1}-\by^k\|^2 \nonumber \\
&    & + 2 \alpha (\bx^{k+1}-\by^k)^\top ( (1-\theta)\bz^k+\theta \by^k - \bx^* ) . \label{CP-norm}
\end{eqnarray}
Take $C>0$, whose value will be determined in a moment. Multiplying by $C$ on both sides of \eqref{CP-norm}, and adding the resulting inequality to \eqref{CP-b} we obtain,
\begin{eqnarray*}
& & F(\bx^{k+1}) - F(\bx^*) + C \|\bz^{k+1} - \bx^*\|^2 \nonumber \\
&\le& (1-\theta) \left( F(\bx^{k}) - F(\bx^*) + C \|\bz^{k} - \bx^*\|^2 \right) \\
&+& \left(\theta C -\frac{(\tau\mu - \xi)\theta}{2}\right) \|\bx^*-\by^k\|^2 + \left( \frac{rL + \xi}{2} + \alpha^2 C \right) \|\bx^{k+1}-\by^k\|^2 \nonumber \\
&+& r L(\bx^{k+1}-\by^k)^\top \left[(1-\theta) \bx^k + \theta \bx^* - \bx^{k+1} + \frac{2\alpha C}{rL} \left((1-\theta)\bz^k+\theta \by^k - \bx^*\right) \right] \nonumber \\
&=& (1-\theta) \left( F(\bx^{k}) - F(\bx^*) + C \|\bz^{k} - \bx^*\|^2 \right) + r L \|\bx^{k+1}-\by^k\|^2 \nonumber \\
& & + r L(\bx^{k+1}-\by^k)^\top \left[(1-\theta) \bx^k  - \bx^{k+1} + \theta \left((1-\theta)\bz^k+\theta \by^k \right) \right] \nonumber \\
&=& (1-\theta) \left( F(\bx^{k}) - F(\bx^*) + C \|\bz^{k} - \bx^*\|^2 \right),
\end{eqnarray*}
if we choose the parameters to take the values
\begin{equation}\label{eqn:thetaANDalpha}
C:=\frac{\tau \mu - \xi}{2},\,\, \alpha:=\sqrt{\frac{r L - \xi}{\tau \mu - \xi}},\,\, \theta := \frac{\sqrt{(\tau\mu - \xi)(rL -\xi)}}{rL}.
\end{equation}
Observe $\theta \in (0,1)$. Since $\tau \leq r$, $\mu \leq L$ and $\tau \mu - \xi > 0$, it follows
\[0 < \tau \mu - \xi \leq  rL - \xi < rL.\]
Note in the last step above, we used \eqref{y-def2},
\[
\by^k = \frac{1}{1+\theta}\, \bx^k + \frac{\theta}{1+\theta} \, \bz^k.
\]
Therefore, $(1-\theta) \bx^k  + \theta \left((1-\theta)\bz^k+\theta \by^k \right) = \by^k$, and so
\[
r L(\bx^{k+1}- \by^k)^\top \left[(1-\theta) \bx^k  - \bx^{k+1} + \theta \left((1-\theta)\bz^k+\theta \by^k \right) \right] = - r L \| \bx^{k+1}-\by^k\|^2.
\]
Hence, we have for all $k\geq 0$,
\begin{multline}
F(\bx^{k+1}) - F(\bx^*) + \left(\frac{\tau \mu - \xi}{2}\right) \|\bz^{k+1} - \bx^*\|^2 \le \nonumber \\ \left(1- \frac{\sqrt{(\tau\mu - \xi)(rL -\xi)}}{rL} \right) \left( F(\bx^{k}) - F(\bx^*) +\left( \frac{\tau \mu - \xi}{2}\right) \|\bz^{k} - \bx^*\|^2 \right),
\end{multline}
and so by induction starting from the $k$-th iteration we obtain the result. \qed
\end{proof}

\begin{remark}\label{rem:GFISTA}
{\rm
If we have $\bB(\bx) = \bx$, then assumptions (A1) - (A3) hold trivially. Also, we have $\xi = 0$, $ r =1$ and $\tau = 1$ ensuring $\tau \mu - \xi > 0$ which gives us the convergence result,
\[
F(\bx^{k}) - F(\bx^*) \le \left( 1 -\sqrt{\frac{\mu}{L}}\right)^k \left( F(\bx^{0}) - F(\bx^*) + \frac{ \mu}{2}\|\bz^{0} - \bx^*\|^2 \right).
\]
If $\bz^0 = \bx^0$, then this is the same convergence result produced by GFISTA \cite{CC19,CP16} which is an accelerated variant of FISTA for \eqref{eqn:ACCO}. Therefore, {C-FISTA} generalizes GFISTA into the broader problem class given by the composite optimization model \eqref{eqn:genCCO}.
}
\end{remark}

\begin{remark}\label{rem:scaling}
{\rm
{The condition $\tau\mu-\xi>0$ essentially requires that the product of the first-order derivatives of $H$ and the second-order derivative of $\bB$ should not exceed the product of their curvatures. In some cases, this condition is easy to satisfy by variable-transformation, or scaling. For example, when $\bB(\bx)$ is homogeneous of degree $d_1>0$, i.e.\ $\bB(\lambda \bx)=\lambda^{d_1} \bB(\bx)$ for any $\bx$ and $\lambda >0$, and $R(\bx)$ is homogeneous with degree $d_2>0$, i.e.\ $R(\lambda \bx)=\lambda^{d_2} R(\bx)$. Then for any $\lambda>0$ one can scale the variable as $\bx':=\lambda^{\frac{1}{d_1}}\bx$, and the problem is turned equivalently into minimizing $H(\bB(\bx')) + R(\bx') = H(\lambda \bB(\bx)) + \lambda^{\frac{d_2}{d_1}} R(\bx)$. After the above change of variables, the new objective has the modified parameters $\tau:=\lambda^2 \tau$, $\mu:=\mu$, and $\xi:=\lambda \xi$. Therefore, in this situation by appropriately choosing $\lambda>0$ one can always satisfy the condition $\tau\mu-\xi>0$. }
}
\end{remark}

As mentioned previously in Section \ref{sec:assumptions}, the global nature of assumptions (A2) and (A3) are in many instances stronger than necessary. From the proof of Theorem \ref{thm:iter-complexity}, we see these assumptions only need to hold locally  about the $\bx^k$ and $\by^k$ sequences generated by C-FISTA. Enforcing the conditions to hold on the span of the constraint set ensures global linear convergence; however, even when the assumptions are not strictly satisfied on the span, C-FISTA still proves to be effective in practice.  In Section \ref{numerical results}, we will demonstrate the effectiveness of C-FISTA on group Lasso and geometric programming models where the assumptions are not strictly met globally but asymptotic linear convergence is still achieved. Ultimately, the success of C-FISTA in these regimes demonstrates the current gap between theory and practice; further research should shrink this separation and provide a direction for future inquiry. 

%============================================================================
\section{Motivation: Lasso, Logistic Regression \& Geometric Programming}\label{Lasso and log reg models}
%============================================================================

A class of motivating composite optimization models for which {C-FISTA} is applicable
%,
 %in the special case $B(x)=x$, 
 are the various Lasso formulations. The standard Lasso model originating from Tibshirani \cite{T96},
\[
\begin{array}{lll}
& \min & \frac{1}{2} \| \bA \bx - \bb \|^2 + \gamma \| \bx \|_1 \\
& \mbox{s.t.} & \bx \in \RR^n,
\end{array}
\]
is solvable via {C-FISTA}. All the assumptions for Algorithm \ref{alg:C FISTA alg} will be satisfied when $\bA$ is of full-column rank, but, as we will see in Section \ref{undet_GL}, C-FISTA demonstrates linear convergence when $\bA$ is not of full-rank and the assumptions are almost but not completely met. The results from Section \ref{C-FISTA} further guarantee global linear convergence to the optimal solution even when there exists a constraint on $\bx$. For example, Algorithm \ref{alg:C FISTA alg} can solve the following constrained Lasso model,
\[
\begin{array}{ll}
\min & \frac{1}{2} \| \bA \bx - \bb \|^2 + \gamma \| \bx \|_1 \\
\mbox{s.t.} & l_i \le x_i \le u_i,\,\, i=1,2,...,n.
\end{array}
\]
The utility of the Lasso model is detailed extensively in the literature; see \cite{JOV9,QSG13,YLY11,YL7,ZZST20}. The success of the basic model spawned the development of several Lasso variants including: the group Lasso (GL) model \cite{QSG13,YL7},
\[
\begin{array}{lll}
&\min & \frac{1}{2} \| \bA \bx - \bb \|^2 + \gamma \sum_{j\in \GG}\| \bx(j) \| \\
& \mbox{s.t.} & \bx \in \RR^n,
\end{array}
\]
where $\bx(j)$ is a subvector of $\bx$ corresponding to the indices in $j \in \GG$ where $\GG$ disjointly partitions $\{1,2,\hdots, n\}$, the sparse-group Lasso (SGL) model \cite{SFHT13,ZZST20},
\[
\begin{array}{lll}
&\min & \frac{1}{2} \| \bA \bx - \bb \|^2 + \gamma_1 \sum_{j\in \GG}\| \bx(j) \| + \gamma_2 \|\bx\|_1 \\
& \mbox{s.t.} & \bx \in \RR^n,
\end{array}
\]
and the overlapping group and overlapping sparse-group Lasso (OSGL) models \cite{JOV9,YLY11,ZZST20} which are of identical form to the group Lasso and sparse-group Lasso models respectively, expect $\GG$ no longer is a disjoint partition of $\{1,\hdots, n\}$. As we will see in Section~\ref{duality-subproblems}, the subproblems required in {C-FISTA} for the Lasso models are solvable in closed form or by an efficient subroutine in the case of the overlapping group/sparse-group Lasso models with global linear convergence achieved in each variant.

Another model where {C-FISTA} has demonstrated linear convergence is the sparse-group logistic regression (SGLR) model \cite{SLEP,MGB8},
\[
\begin{array}{lll}
&\min &\frac{1}{m} \sum_{i=1}^{m} \ln\left( 1 +\exp\left( -y_i \left( \ba_i^\top \bx + b \right) \right) \right) + \gamma_1 \sum_{j \in \GG} \| \bx(j) \| + \gamma_2 \| \bx\|_1 \\
& \mbox{s.t.} & \bx \in \RR^n, b \in \RR,
\end{array}
\]
where $\by \in \RR^n$ is a vector containing only entries of $\pm 1$. The sparse-group logistic regression model has applications in various machine learning models especially in the area of classification, and the subproblems for the logistic regression model are in the same format as the sparse-group Lasso model detailed extensively in Section \ref{duality-subproblems}. 

{
A final motivating model is geometric programming. Geometric programming is a useful modeling paradigm which has numerous applications especially in electrical circuit design. A full description and tutorial on these models is provided by Boyd {\it et al.} in \cite{GP}. A regularized subclass of geometric programs solvable via C-FISTA is,
\[
\begin{array}{lll}
&\min & \;\; \sum_{k=1}^{K} c_k x_1^{a_{k1}} x_2^{a_{k2}}\hdots x_n^{a_{kn}} + R(\bx), \\
&\mbox{s.t}&\;\; \bx \in \XX, \nonumber 
\end{array}
\]
where $c_k > 0$, $\mathbf{a_{k}} := (a_{k1}, \hdots, a_{kn})^\top \leq  0$ for all $k$, and $\XX$ is a closed and bounded convex subset of $\RR^n_{++}$. The condition $\mathbf{a_k} \leq 0$ ensures the model is convex. In solving this model, we decompose the objective such that $\bB(\bx) = \left( \ln(x_1), \hdots, \ln(x_n) \right)^\top$. Allowing non-linear and non-convex component functions for $\bB$ demonstrates the novelty of C-FISTA to escape strictly affine compositions showcasing a level of extended generality. 
}

Before discussing how to compute subproblems for the Lasso models, we first discuss Fenchel duality in Section \ref{Fenchel Duality} and describe how {C-FISTA} can be utilized via a dual approach to solve a general constrained composite model.

%============================================================================
\section{A Dual Formulation and Algorithmic Approach}\label{Fenchel Duality}
%============================================================================

In this section, we outline the construction of a dual algorithm using Fenchel duality which generates an approximate primal solution from the dual solution at every iterate. To elucidate this approach we first provide a brief overview of standard convex analysis results regarding conjugate functions and Fenchel duality \cite{RF70,Roos20}. Unless otherwise stated, we use the standard notation of Rockafellar \cite{RF70}. For generality, we will consider the following constrained composite model:
\begin{eqnarray}\label{eqn:genFM}
     & \min &\;\; h_0(\bx)+h_1(\bx)+\cdots+h_\ell(\bx) \\
     & \mbox{s.t.} &\;\; \bx \in \CC_j \subseteq \RR^n,\, j=1,2,...,r, \nonumber
\end{eqnarray}
where the $h_i$'s are proper convex functions on $\RR^n$ and the $\CC_j$'s are closed convex subsets of $\RR^n$. To begin, let us introduce some basic notions in convex analysis.

 Let $S\subseteq \RR^n$ be a closed convex set. The {\it support function} of $S$ is defined as,
\[
f_S(\by) : = \max_{\bx \in S} \by^\top \bx,
\]
and the {\it polar set} of $S$ is,
\[
S^\circ:=\{ \by \mid \by^\top \bx \le 1,\, \forall \bx \in S\}.
\]
Thus, we can succinctly write, $S^\circ = \{ \by \mid f_S(\by) \le 1\}$. Let $g$ be a convex function whose domain is contained in $S$. The {\it conjugate of $g$} is defined as,
\[
g^*(\by) := \sup_{\bx \in S} \left[ \by^\top \bx - g(\bx) \right],
\]
where the function is assumed to take value $+\infty$ anywhere outside its domain. In deriving the Fenchel dual of \eqref{eqn:genFM} we need some standard results from convex analysis.

\begin{lemma} \label{lemma1}
Suppose $f$ is a proper convex function, then
\begin{equation}
-f^*(\mathbf{0}) = \inf_{\bx} \; f(\bx).
\end{equation}
\end{lemma}

\begin{lemma} \label{lemma2}
(Theorem 16.4 \cite{RF70}) Suppose $g_1,...,g_r$ are proper convex functions on $\RR^n$ and the intersection of the relative interiors of the domains of the $g_i$'s is non-empty, then
\begin{equation} \label{infimal}
\left( g_1 + \cdots + g_r \right)^*(\bx) = \inf_{\by_1+\cdots+\by_r=x } \sum_{j=1}^r g_j^*(\by_j).
\end{equation}
\end{lemma}

\begin{lemma} \label{lemma3}
For a convex set $S$, suppose its indicator function is defined as,
\[
I_{S}(x) := \left\{ \begin{array}{ll}
0, & \mbox{if $\bx\in S$}; \\
+\infty, & \mbox{if $\bx\not \in S$}.
\end{array} \right.
\]
Then, the conjugate of the indicator function is simply the support function of $S$;
\[
I_{S}^* = f_S.
\]
\end{lemma}

The proofs of Lemmas \ref{lemma1} and \ref{lemma3} follow directly from the definitions of conjugate and support functions. Thus, rewriting \eqref{eqn:genFM} in the equivalent form,
\begin{eqnarray}
     & \min & h_0(\bx)+h_1(\bx)+\cdots+h_\ell(\bx) + I_{\CC_1}(\bx) + \hdots + I_{\CC_r}(\bx) \nonumber  \\
     & \mbox{s.t.} & \bx \in \RR^n, \nonumber
\end{eqnarray}
and utilizing Lemmas \ref{lemma1}, \ref{lemma2}, and \ref{lemma3}, the Fenchel dual of \eqref{eqn:genFM} is:
\begin{eqnarray}\label{eqn:dualFM}
           & \min & h_0^*(-\by_1-\cdots - \by_\ell- \by_{\ell+1} - \cdots - \by_{\ell+r}) \\
           &      &  +h_1^*(\by_{1}) + \cdots + h^*_\ell(\by_{\ell}) + f_{\CC_1}(\by_{\ell+1}) +\cdots + f_{\CC_r}(\by_{\ell+r}) \nonumber  \\
& \mbox{s.t.} & \by_j \in \RR^n, \, j=1,2,...,\ell+r. \nonumber
\end{eqnarray}
Note, for the sake of presentation we have written the dual problem \eqref{eqn:dualFM} in such a manner that the optimal solutions of \eqref{eqn:genFM} and \eqref{eqn:dualFM} differ by a negative sign. To better understand the link between the primal and dual models, we demonstrate how the optimality conditions of \eqref{eqn:genFM} and \eqref{eqn:dualFM} are linked. Assuming the conjugate functions of the $h_i$'s ($i=0,1,...,\ell$) are differentiable  and denoting $\by^*_j$ for $j=1,\hdots, \ell +r$ to be optimal for \eqref{eqn:dualFM}, we have the first-order optimality conditions,
\begin{equation} \label{dual-optimality}
\left\{
\begin{array}{ll}
-\nabla h_0^*(-\sum_{j=1}^{\ell+r} \by_j^*) + \nabla h_i^*(\by^*_i) = 0, & i=1,...,\ell \\
\nabla h_0^*(-\sum_{j=1}^{\ell+r} \by_j^*) \in \partial f_{\CC_j} (\by^*_{\ell+j}) , & j=1,...,r.
\end{array}
\right.
\end{equation}

Denoting $\bx^*:= \nabla h_0^*(-\sum_{j=1}^{\ell+r} \by_j^*)$, the optimality conditions imply $\bx^*=\nabla h_i^*(\by^*_i) $ for $i=1,\hdots,\ell$, and so by Lemma 5 of \cite{Z18}, $\nabla h_i(\bx^*)= \by^*_i$, for $i=1,...,\ell$. Similarly, because $f_{\CC_j}^* = I_{\CC_j}$,
if $ \bx^* \in \partial f_{\CC_j} (\by^*_{\ell+j})$, then $\by^*_{\ell+j} \in \partial I_{\CC_j}(\bx^*)$, which means that $\by^*_{\ell+j}$ is a normal direction at $\bx^*$:
\[
(\by^*_{\ell+j})^\top (\bx-\bx^*) \le 0, \mbox{ for all $\bx\in \CC_j$}.
\]
Furthermore, by \eqref{dual-optimality}
\begin{equation} \label{dual-opt2}
\nabla h_0(\bx^*) = -\sum_{j=1}^{\ell+r} \by_j^* = -\sum_{i=1}^\ell \nabla h_i(\bx^*) - \sum_{j=1}^r \by^*_{\ell+j}.
\end{equation}
Thus, the first-order optimality conditions for the dual problem implies \eqref{dual-opt2}, which is the optimality condition for the primal problem \eqref{eqn:genFM}. To see this, note that
\[
\nabla h_0(\bx^*) + \sum_{i=1}^\ell \nabla h_i(\bx^*) = - \sum_{j=1}^r \by^*_{\ell+j},
\]
and so,
\[
\left( \nabla h_0(\bx^*) + \sum_{i=1}^\ell \nabla h_i(\bx^*) \right)^\top \hspace{-0.1in} ( \bx - \bx^*) = - \sum_{j=1}^r \left(\by^*_{\ell+j} \right)^\top \hspace{-0.05in}( \bx - \bx^*) \ge 0,\,\, \forall \bx \in \cap_{j=1}^r \CC_j.
\]

We now present an important relationship between the primal solution induced by the dual iterations which outlines a dual approach to solving \eqref{eqn:genFM}.
\begin{proposition}\label{dual_prop}
Suppose that $h_0$ is a strongly convex function with strong convexity parameter $\mu$ and gradient Lipschitz constant $L$. {Let $\{\by^k\}$ be a  sequence converging to the dual solution $\by^*$. For each $\by^k$ in the dual sequence, we recover a primal solution
$\bx^k:= \nabla h_0^*(-\sum_{j=1}^{\ell+r} \by_j^k)$ such that,}
\[
\frac{1}{L} \left\| \sum_{j=1}^{\ell+r} \left( \by^k_j - \by^*_j \right) \right\| \le \| \bx^k - \bx^*\| \le \frac{1}{\mu} \left\| \sum_{j=1}^{\ell+r} \left( \by^k_j - \by^*_j \right) \right\|.
\]
\end{proposition}

\begin{proof}
By the Fenchel duality relation, we know that $h_0^*$ is also a strongly convex function with strong convexity parameter $1/L$ and gradient Lipschitz constant $1/\mu$ (Theorem 1 \cite{Z18}). {Further, by the definition of $\bx^k$ and previously stated results,}
$ \nabla h_0(\bx^k)= -\sum_{j=1}^{\ell+r} \by_j^k$. Therefore, $ \nabla h_0(\bx^k) - \nabla h_0(\bx^*) = \sum_{j=1}^{\ell+r} ( \by_j^*- \by_j^k)$, and so the gradient Lipschitz condition on $h$ yields,
\[
\left\| \sum_{j=1}^{\ell+r} \left( \by^k_j - \by^*_j \right) \right\| = \| \nabla h_0(\bx^k) - \nabla h_0(\bx^*) \| \le L \|\bx^k - \bx^*\|.
\]
On the other hand, the gradient Lipschitz condition on $h^*$ states,
\[
\| \bx^k- \bx^*\|= \left\|\nabla h_0^*\left(-\sum_{j=1}^{\ell+r} \by_j^k\right) - \nabla h_0^*\left(-\sum_{j=1}^{\ell+r} \by_j^*\right)\right\| \le \frac{1}{\mu} \left\| \sum_{j=1}^{\ell+r} \left( \by^k_j - \by^*_j \right) \right\|. 
\]
\qed
\end{proof}
Therefore, the above result gives us the framework to develop a general dual approach to solving \eqref{eqn:genFM}. Applying any algorithm to solve the dual problem \eqref{eqn:dualFM} we can recover a primal solution at each iteration; furthermore, the rate of convergence on the primal side for any algorithm will only differ by a fixed constant, $1/\mu$, in-comparison to the rate of convergence on the dual side.

The presented dual formulation is not novel; Han and Lou \cite{HL88} and Fukushima {\it et al.}~\cite{FHN96} have very similar derivations of the dual and many algorithms such as Dykstra's projection algorithm \cite{GM89} utilize Fenchel duality. The clarity provided by the previous result is that, under the assumption $h_0$ is strongly convex with gradient Lipschitz, any algorithm which generates a converging dual solution will generate a primal solution with the same rate of convergence up to a constant factor. {This is directly relevant to our proposed algorithm because the dual \eqref{eqn:dualFM} is solvable via C-FISTA. Let $H: \RR^n \rightarrow \RR,\; \bB: \RR^{n(\ell +r)} \rightarrow \RR^n$, and $R: \RR^{n(\ell + r)} \rightarrow \RR$ such that,
\begin{align}
H(\bx) &= h_0^*(\bx), \nonumber \\
\bB(\bx) &= \left[ -\mathbf{I}_n \;| \; \hdots \; | \; - \mathbf{I}_n\right]\bx, \nonumber \\
R(\bx) &= h_1^*(\mathbf{E}_1 \bx) + \cdots + h^*_\ell(\mathbf{E}_{\ell}\bx) + f_{\CC_1}(\mathbf{E}_{\ell+1}\bx) +\cdots + f_{\CC_r}(\mathbf{E}_{\ell+r}\bx), \nonumber 
\end{align}
%\]
%and,
%\[
%R(x) = h_1^*(E_1 x) + \cdots + h^*_\ell(E_{\ell}x) + f_{\CC_1}(E_{\ell+1}x) +\cdots + f_{\CC_r}(E_{\ell+r}x),
%\end{algin}
%\[
%H(x) = h_0^*(x), \;\;\; B(x) = \left[ -I_n \;| \; \hdots \; | \; - I_n\right]x, 
%\]
%and,
%\[
%R(x) = h_1^*(E_1 x) + \cdots + h^*_\ell(E_{\ell}x) + f_{\CC_1}(E_{\ell+1}x) +\cdots + f_{\CC_r}(E_{\ell+r}x),
%\]
where $\mathbf{I}_n$ is the $n\times n$ identity matrix and $\mathbf{E}_k$ is the $1 \times (\ell + r)$ block matrix with the identity matrix $\mathbf{I}_n$ in the k-th position, e.g. $\mathbf{E}_2 = \left[ \mathbf{0} \; | \; \mathbf{I}_n \; | \; \mathbf{0} \; | \; \hdots \; | \; \mathbf{0} \right] \in \RR^{n\times n(\ell+r)}$, then we can rewrite \eqref{eqn:dualFM} as,
\[
\begin{array}{lll}
& \min & H(\bB(\bx)) + R(\bx) \\
& \mbox{s.t.} & \bx \in  \RR^{n(\ell +r)}.
\end{array}
\]
With this decomposition it is straightforward to specify the constants: $\mu'$, $L'$, $\tau'$, $r'$ and $\xi'$ for C-FISTA. Assuming $h_0$ is strongly convex with parameter $\mu$ and gradient Lipschitz with constant $L$, it follows $\mu' = 1/L$ and $L' = 1/\mu$. The linearity of $\bB$ yields $\xi' = 0$, and the simple structure of the mapping gives $r' = \|\left[ -\mathbf{I}_n \;| \; \hdots \; | \; - \mathbf{I}_n\right]\|^2_2 = (\ell + r)$. Lastly, $\tau' = 0$ since the matrix defining $\bB$ is singular; however, though $\tau' = 0$ does not strictly satisfy assumption (A2), in practice C-FISTA will often converge for a range of small $\tau'$ values. This is demonstrated in Section \ref{undet_GL} through multiple numerical experiments on an underdetermined group Lasso model and follows from our previous discussion on the overly strict nature of the assumptions which ensure global linear convergence. 

A final note must be made in-regards to the proximal mapping step which would be required to solve the dual formulation. Computing $\mbox{\rm Prox}_{ R}(\cdot)$ in this case would depend on the definitions of $h_i^*$ and $f_{C_i}$. This could prove difficult; however, $R$ presents a natural decomposition such that only the individual proximal mappings of the functions $h_i^*$ and $f_{C_i}$ would be necessary. This substantially simplifies the procedure enabling tractable computations in many instances. 

Therefore, we see that \eqref{eqn:genFM} under the proper convexity assumptions has a dual which in many cases is solvable via C-FISTA, and Proposition \ref{dual_prop} demonstrates the primal solution is recoverable from the dual solution without losing linear convergence.}

%
%
% Connecting this result back to {C-FISTA} in the special case $B(x) = x$, we can solve the dual problem \eqref{eqn:dualFM} and obtain a primal solution from Proposition \ref{dual_prop}. Therefore, any primal problem \eqref{eqn:genFM} can be solved with {C-FISTA} assuming the required proximal mappings are computable and the necessary assumptions for Theorem \ref{thm:iter-complexity} are maintained.
Besides using Fenchel duality {to directly solve the dual model with C-FISTA and recover the primal solution, we can also utilize Fenchel duality theory to solve the intermediate subproblems in the primal implementation of C-FISTA.} In the next section, we describe this process using the sparse-group Lasso model as an example.

%============================================================================
\section{Solving Subproblems with Fenchel Duality} \label{duality-subproblems}
%============================================================================
One example of how duality can be leveraged to apply {C-FISTA} can be seen in how we solve the subproblems for the group and sparse-group Lasso models. In the sparse-group Lasso model, the main subproblem which needs to be solved to apply {C-FISTA} is of the form,
\[
\begin{array}{lll}
(S) & \min & \frac{1}{2} \| \bx - \bd \|^2 + \gamma_1 \| \bx \| + \gamma_2 \| \bx \|_1 \\
& \mbox{s.t.} & \bx \in \RR^n,
\end{array}
\]
with $\gamma_1, \gamma_2 \geq 0$. Utilizing Lemmas \ref{lemma1} and \ref{lemma2} and standard conjugate functions \cite{Roos20}, we see the Fenchel dual of $(S)$ is,
\[
\begin{array}{lll}
(DS)  & \min & \frac{1}{2} \| \bd + \by_1 + \by_2 \|^2  \\
& \mbox{s.t.} & \|\by_1\| \le \gamma_1, \\
&             & \|\by_2\|_\infty \le \gamma_2 .
\end{array}
\]
When $\by_2$ is fixed, consider the simple projection problem,
\[\min_{\|\by_1\|\le \gamma_1} \,\, \| \bd+\by_2 + \by_1\|,
\]
whose solution is explicit:
\[
\by_1^* = \left\{ \begin{array}{ll} -\bd-\by_2, & \mbox{ if $\|\bd+\by_2\| \le \gamma_1$}; \\ - \gamma_1 (\bd+\by_2) / \|\bd+\by_2\|, & \mbox{ if $\|\bd+\by_2\| > \gamma_1$}. \end{array} \right.
\]
Therefore,
\[
\|\bd+\by_2+\by_1^*\| = \left\{ \begin{array}{ll} 0, & \mbox{ if $\|\bd+\by_2\| \le \gamma_1$}; \\ \|\bd+\by_2\| - \gamma_1 , & \mbox{ if $\|\bd+\by_2\| > \gamma_1$}. \end{array} \right.
\]
Thus, the dual problem $(DS)$ is equivalent to $\min_{\|\by_2\|_\infty \le \gamma_2} \,\, (\|\bd+\by_2\| - \gamma_1)_+ $, which can be solved through the associated model,
\[
\begin{array}{lll}
(DS)' & \min & \frac{1}{2} \| \bd + \by_2 \|^2  \\
& \mbox{s.t.} & \|\by_2\|_\infty \le \gamma_2 .
\end{array}
\]
Since the individual components of $\by_2$ are decoupled in $(DS)'$, the above problem can be reduced to solving $n$ 1-dimensional problems,
\[
\begin{array}{ll}
\min        & (d_i+t)^2 \\
\mbox{s.t.} & |t| \le \gamma_2,
\end{array}
\]
for $i=1,2,...,n$. Observe that the solutions for the above 1-dimensional models can be found using the following thresholding operator:
\[
{\Th}_{\gamma_2} (t):= \left\{
\begin{array}{rl}
-\gamma_2 , & \mbox{if $ t < - \gamma_2$}; \\
t , & \mbox{if $-\gamma_2 \le t \le \gamma_2$}; \\
\gamma_2 , & \mbox{if $ t > \gamma_2$}. \\
\end{array} \right.
\]
Therefore, a solution for $(DS)$ is $\by_2^*=\textbf{Th}_{\gamma_2}(-\bd)$, where $(\textbf{Th}_{\gamma_2}(-\bd))_i = \Th_{\gamma_2}(-d_i)$ for all $i$, and,
\[
\by_1^* = \left\{ \begin{array}{ll} -\bd-\textbf{Th}_{\gamma_2}(-\bd), & \mbox{ if $\|\bd+\textbf{Th}_{\gamma_2}(-\bd)\| \le \gamma_1$}; \\ - \gamma_1 (\bd+\textbf{Th}_{\gamma_2}(-\bd)) / \|\bd+\textbf{Th}_{\gamma_2}(-\bd)\|, & \mbox{ if $\|\bd+\textbf{Th}_{\gamma_2}(-\bd)\| > \gamma_1$}. \end{array} \right.
\]
After solving $(DS)$ with the optimal solution $(\by_1^*,\by_2^*)$, one recovers the optimal solution to $(S)$ as $\bx^*=\bd+\by_1^*+\by_2^*$ using the results from Section \ref{Fenchel Duality}. In particular,
\[
\bx^* = \left\{ \begin{array}{ll}
0, & \mbox{ if $\|\bd+\textbf{Th}_{\gamma_2}(-\bd)\| \le \gamma_1$}; \\
\frac{(\|\bd+\textbf{Th}_{\gamma_2}(-\bd)\|- \gamma_1)}{\|\bd+\textbf{Th}_{\gamma_2}(-\bd)\|} \left(\bd+\textbf{Th}_{\gamma_2}(-\bd)\right), & \mbox{ if $\|\bd+\textbf{Th}_{\gamma_2}(-\bd)\| > \gamma_1$}.
\end{array} \right.
\]
Hence, by Fenchel duality we obtain a closed form solution to $(S)$ yielding an exact solution to the required subproblems to apply {C-FISTA} to the sparse-group Lasso model. Similar approaches can be taken to solve other potential subproblems which arise when applying {C-FISTA}. In the next section, we present the algorithms and numerical experiments for solving the models discussed in Section \ref{Lasso and log reg models}.

%============================================================================
\section{Numerical Experiments} \label{numerical results}
%============================================================================

To demonstrate the practicality and efficiency of {C-FISTA}, we conducted numerous experimental tests on group Lasso, sparse-group Lasso, overlapping sparse-group Lasso, sparse-group logistic regression,  and regularized geometric programming models. The overall structure of this section is as follows: Subsection \ref{group_Lasso} details how to apply {C-FISTA} to solve the Lasso formulations; { Subsection \ref{Lasso_Exp} contains the results of the numerical experiments conducted on the Lasso models; Subsections \ref{SGLR} and \ref{SGLR_Exp} present the solution procedure and numerical results for the sparse-group logistic regression model; Subsections \ref{GeoProg} and \ref{GeoProg_exp} describe the application of C-FISTA to a set of regularized geometric programs and presents some numerical results.
}

%============================================================================
\subsection{Lasso Models}\label{group_Lasso}
%============================================================================
{In this section we state how to apply {C-FISTA} to solve the various Lasso models described in Section \ref{Lasso and log reg models}. {In order to apply C-FISTA, the practitioner must select a decomposition of the objective function, i.e. the user must define $H$ and $\bB$ in \eqref{eqn:genCCO}. For the sake of exposition, in Sections \ref{Group_Lasso_process},  \ref{SGL_process}, and \ref{OSGL_process}, we decompose $H(\bB(\bx)) = \frac{1}{2}\|\bA\bx - \bb\|^2$ as $H(\bx) = \frac{1}{2}\|\bA\bx - \bb\|^2$ and $\bB(\bx) = \bx$. By doing this we have $\xi = 0$, $r=1$, and $\tau = 1$. Therefore, the strong convexity constant $\mu$ and gradient Lipschitz constant $L$ for $H$ are the only parameters which must be estimated to apply C-FISTA. This was the convention chosen for the algorithms presented in these sections as it enables more concise algorithmic descriptions; however, another viable decomposition is to have $H(\bx) = \frac{1}{2}\|\bx - \bb\|^2$ and $\bB(\bx) = \bA\bx$. This decomposition is utilized in Section \ref{undet_GL}. Under this decomposition, one readily obtains the strong convexity and gradient Lipschitz constants as $\mu = L = 1$. As for the other constants, the linearity of $\bB$ yields $\xi = 0$ and $r = \lambda_{\text{max}}(\bA^\top \bA)$. Hence, the only constant to estimate is $\tau$ from assumption (A2). The ability to chose the decomposition for a particular problem is a benefit of C-FISTA. This freedom enables the practitioner to select the most beneficial and/or convenient decomposition for their model of interest. 

We now begin our discussion with how to apply C-FISTA to the group Lasso formulation. 
}

%============================================================================
\subsubsection{C-FISTA for Group Lasso}\label{Group_Lasso_process}
%============================================================================
Applying {C-FISTA} to solve any model hinges on computing the proximal mapping in \eqref{eqn:x-def}. This key subproblem in the group Lasso model is of the form,
\begin{equation} \label{eqn:GL_subprob}
\begin{array}{ll}
\min & \frac{1}{2} \| \bx - \bd \|^2 + \gamma \| \bx \| \\
\mbox{s.t.} & \bx \in \RR^n,
\end{array}
\end{equation}
which by the first-order optimality conditions has the simple closed-form solution,
\[
\bx^* = \begin{cases} 0, & \text{ if } \| \bd\| \leq \gamma; \\ \left( \frac{ \|\bd\| - \gamma}{\|\bd\|}\right)\bd, & \text{ if } \|\bd\| > \gamma . \end{cases}
\]
With the solution to the group Lasso subproblems \eqref{eqn:GL_subprob}, we write down how to solve the group Lasso model with {C-FISTA} in Algorithm \ref{alg:C_FISTA_GL} {where $H(\bx) = \frac{1}{2}\|\bA\bx - \bb\|^2$ and $\bB(\bx) = \bx$.}

\begin{algorithm}
\caption{C-FISTA for $(GL)$}
\begin{algorithmic}\label{alg:C_FISTA_GL}
{ \STATE \textbf{Input:} Constants $\mu$ and $L$; penalty parameter $\gamma > 0$; vector partition $\mathcal{G}$.}
\STATE \textbf{Step 0.} Choose any $(\bx^0,\bz^0) \in \mathbb{R}^n \times \mathbb{R}^n$. Let $k:=0$.

\STATE \textbf{Step 1.} Let,
\[ \by^k := \frac{1}{1 + \theta} \bx^k + \frac{\theta}{1 + \theta} \bz^k. \]

%\STATE \textbf{Step 2.}  For $j=1,\hdots, J$:
\STATE \textbf{Step 2.}  {For each $j \in \mathcal{G}$:}
\[ \bx(j)^{k+1} := \begin{cases} 0, & \|\bd(j)\| \leq \gamma/L; \\ \left( \frac{\|\bd(j)\| - \gamma/L}{\|\bd(j)\|} \right) \bd(j), & \|\bd(j)\| > \gamma/L; \end{cases}\]
with $\bd:= \by^k - \frac{1}{L}\nabla H(\by^k) = \by^k - \frac{1}{L}\bA^\top (\bA\by^k -\bb )$ and corresponding subvectors $\bd(j)$.

\vspace{0.05in}
\STATE \textbf{Step 3.}  \[ \bz^{k+1}:= (1-\theta) \bz^k + \theta \by^k + {\alpha} \left(\bx^{k+1} - \by^k \right),\]
{where $\theta = \sqrt{\mu/L}$ and $\alpha = \sqrt{L/\mu}$.}\vspace{3mm}
\STATE \textbf{Step 4.}  Let $k:= k+1$; return to Step 1 until convergence.
\end{algorithmic}
\end{algorithm}

Practical application of Algorithm \ref{alg:C_FISTA_GL} requires the user to determine bounds for the Lipschitz constant and strong convexity constant of $H$. For the group Lasso models $H(\bx):=\frac{1}{2}\|\bA\bx-\bb\|^2$ making such bounds readily available. By the definition of strong convexity we see $\mu = \lambda_{\min}(\bA^\top \bA)$. Similarly, using the definition of the Lipschitz constant, we see that $L \leq \lambda_{\max}(\bA^\top \bA)$ where $\lambda_{\min}(\bA^\top \bA)$ and $\lambda_{\max}( \bA^\top \bA)$ are the smallest and largest eigenvalues of $\bA^\top \bA$ respectively. For other strongly convex functions $H$ on which {C-FISTA} is applicable, tight bounds for $L$ and $\mu$ might be unavailable. In these situations conservative estimates for these bounds or a backtracking scheme such as in \cite{BT09} to estimate these parameters would be required. In this paper we do not focus on subroutines to estimate these bounds in our algorithms; however, the backtracking strategies applied in \cite{BT09,CC19,FV19} could similarly be implemented in many instances of {C-FISTA}.

%============================================================================
\subsubsection{C-FISTA for Sparse-Group Lasso}\label{SGL_process}
%============================================================================

The key subproblem in applying {C-FISTA} to solve the sparse-group Lasso model was determined in Section \ref{duality-subproblems}; thus, we can write Algorithm \ref{alg:C_FISTA_SGL} to solve the sparse-group Lasso model with {C-FISTA}. 
\begin{algorithm}
\caption{C-FISTA for $(SGL)$}
\begin{algorithmic}\label{alg:C_FISTA_SGL}
{\STATE \textbf{Input:} Constants $\mu$ and $L$; penalty parameters $\gamma_1, \gamma_2 > 0$; vector partition $\mathcal{G}$.}
\STATE \textbf{Step 0.} Choose any $(\bx^0,\bz^0) \in \mathbb{R}^n \times \mathbb{R}^n$. Let $k:=0$. 

\STATE \textbf{Step 1.} Let
\[ \by^k := \frac{1}{1 + \theta} \bx^k + \frac{\theta}{1 + \theta} \bz^k. \]

\STATE \textbf{Step 2.} {For each $j \in \mathcal{G}$:} define $\mathbf{\Omega}(j) := \|\bd(j)+\textbf{Th}_{\gamma_2/L}(-\bd(j))\|$ and,
\[
\bx(j)^{k+1}\hspace{-0.05in} :=\hspace{-0.03in} \left\{ \begin{array}{ll}
\hspace{-0.02in} 0, & \hspace{-0.05in}\mbox{ if $\mathbf{\Omega}(j) \le \frac{\gamma_1}{L}$}; \\
\hspace{-0.07in}\left(\frac{\|\bd(j)+\textbf{Th}_{\gamma_2/L}(-\bd(j))\|- \gamma_1/L)}{\|\bd(j)+\textbf{Th}_{\gamma_2/L}(-\bd(j))\|}\right) \left(\bd(j)+\textbf{Th}_{\gamma_2/L}(-\bd(j))\right), &\hspace{-0.05in} \mbox{ if $\mathbf{\Omega}(j) > \frac{\gamma_1}{L}$};
\end{array} \right.
\]
with $\bd:= \by^k - \frac{1}{L}\bA^\top (\bA\by^k - \bb )$ and corresponding subvectors $\bd(j)$. 
\STATE \textbf{Step 3.} \[ \bz^{k+1}:= (1-\theta) \bz^k + \theta \by^k + {\alpha} \left( \bx^{k+1} - \by^k \right),\]
{where $\theta = \sqrt{\mu/L}$ and $\alpha = \sqrt{L/\mu}$.}\vspace{3mm}
\STATE \textbf{Step 4.}  Let $k:= k+1$; return to Step 1 until convergence.
\end{algorithmic}
\end{algorithm}

Since $H$ is identical in all of the Lasso models, the same bounds for the Lipschitz and strong convexity constants in Algorithm \ref{alg:C_FISTA_GL} apply for each of the Lasso models. Comparing Algorithms \ref{alg:C_FISTA_GL} and \ref{alg:C_FISTA_SGL}, we note only Step 2 has been updated because the $\by$ and $\bz$-updates are independent of the objective function. The alteration in Step 2 is solely due to the difference in the regularization term in the group and sparse-group Lasso models which alters the proximal mapping \eqref{eqn:x-def}. 

%============================================================================
\subsubsection{C-FISTA for Overlapping Sparse-Group Lasso}\label{OSGL_process}
%============================================================================

The key difference between the overlapping sparse-group Lasso model and those previously discussed is the removal of the prohibition that the groups of variables cannot intersect. Allowing overlapping groups of variables removes the ability to decouple the minimization in \eqref{eqn:x-def} into subproblems over the individual subvectors. In the sparse-group Lasso model, since the groups of variables do not overlap, we were able to show in Section \ref{duality-subproblems} there was a closed form solution to the proximal mappings; however, with overlapping groups the key subproblem for implementing C-FISTA,
\begin{equation}\label{eqn:OSGL_subprob}
\min\; \frac{1}{2} \| \bx- \bd \|^2 + \frac{\gamma_1}{L} \sum_{j \in \GG} \| \bx(j) \| + \frac{\gamma_2}{L} \| \bx\|_1,
\end{equation}
is no longer decomposable and a simple closed form solution is unobtainable; therefore, we must solve this subproblem numerically at each iteration in-order to apply {C-FISTA} to the overlapping sparse-group Lasso model. In \cite{YLY11} the authors apply a variation of FISTA to solve the (OSGL) model. Our approach is distinct from their approach because we accelerate FISTA through the use of two sequences, $\{\by^k\}$ and $\{\bz^k\}$, while they use a single sequence performing a Nesterov-like acceleration. In order to apply their algorithm, the authors of \cite{YLY11} developed an efficient subroutine, {\it overlapping.c}, to numerically solve \eqref{eqn:OSGL_subprob}. Their subroutine is provided in the SLEP software package \cite{SLEP}. In our solving of the (OSGL) model with {C-FISTA}, we apply Algorithm \ref{alg:C_FISTA_SGL} exactly in the same fashion as for the sparse-group Lasso model except when computing $\bx^{k+1}$ in Step 2 we solve \eqref{eqn:OSGL_subprob} with \textit{overlapping.c} from the SLEP software package.

%============================================================================
\subsection{Lasso Numerical Experiments}\label{Lasso_Exp}
%============================================================================
{In this section, we describe the results from the numerical experiments conducted on the Lasso models. In the first set of experiments, the data matrix $A$ has full-column rank. In this setting, all assumptions stated in Section \ref{C-FISTA} are fully met. In practice, however, it is often unrealistic to assume the data matrix for a Lasso model is overdetermined; therefore, we conducted a second set of experiments on the group Lasso model with underdetermined data to demonstrate the applicability of C-FISTA in this setting. 

%============================================================================
\subsubsection{Overdetermined Lasso Models Numerical Tests}\label{lasso_test_1}
%============================================================================
For comparing {C-FISTA}, SLEP, ADMM and FISTA (without backtracking) on the overdetermined Lasso models, we constructed synthetic data sets for testing in a manner similar to the process utilized in \cite{ZZST20}. For our tests of the group and sparse-group Lasso models, we randomly generated a $m \times n$ data matrix $A$ from the standard normal distribution, and we formed three different subvector groups, $\mathcal{G}_1, \mathcal{G}_2,$ and $\mathcal{G}_3$, where the subvectors in $\mathcal{G}_i$ where of sizes $10, 100,$ and 200 respectively. Thus, in our experiments we let $\bA \in \mathbb{R}^{8000 \times 5000}$ and had,
\begin{align}
\mathcal{G}_1 &= \{ (1,\hdots,10), (11, \hdots ,20), \hdots, (4991, \hdots ,5000)\}, \nonumber \\
\mathcal{G}_2 &= \{ (1,\hdots,100), (101, \hdots ,200), \hdots, (4901, \hdots ,5000)\}, \nonumber \\
\mathcal{G}_3 &= \{ (1,\hdots,200), (201,\hdots ,400), \hdots, (4801, \hdots ,5000)\}. \nonumber
\end{align}
The response vectors $\bb^i$ for $i=1,2,3$ were formed as, 
\[ \bb^i = \bA\bx + \delta \cdot\mathbf{\epsilon}, \]
where each component of $\mathbf{\epsilon}$ was drawn from the standard norm distribution, $\delta>0$ provided a scaling factor for the noise, and $\bx(j) = (1,2,\hdots, 10,0,\hdots,0)^\top$ for $j=1,2,\hdots,10$ with $\bx(j) =\mathbf{0}$ for the remaining subvector groups $j=11,\hdots, |\mathcal{G}_i|$. We set the positive penalty parameters $\gamma_1$ and $\gamma_2$ to be equal and of a magnitude to ensure sparse but non-trivial solutions.

For the overlapping sparse-group Lasso model we constructed the data matrix $A$ in the same manner, and let the subvector groups be,
\begin{align}
\mathcal{G}_1 &= \{ (1,\hdots,10), (6,\hdots ,15), \hdots,(4986,\hdots, 4995), (4991, \hdots ,5000)\}, \nonumber \\
\mathcal{G}_2 &= \{ (1,\hdots,100), (51,\hdots ,150), \hdots,(4851,\hdots,4950), (4901, \hdots ,5000)\}, \nonumber \\
\mathcal{G}_3 &= \{ (1,\hdots,200), (101,\hdots ,300), \hdots, (4701,\hdots,4900), (4801, \hdots ,5000)\}. \nonumber
\end{align}
Thus, each of the adjacent subvectors in $\mathcal{G}_i$ for $i=1,2,3$ overlapped with another subvector. The response vectors were constructed as done for the group and sparse-group Lasso models.

%===================================
The overdetermined Lasso experiments were conducted in MATLAB R2021. Implementation of C-FISTA was as detailed in Sections \ref{Group_Lasso_process},  \ref{SGL_process}, and \ref{OSGL_process}. From the SLEP software package \cite{SLEP} we utilized: {\it glLeastR.m}, {\it sgLeastR.m} and {\it overlapping\textunderscore LeastR.m} to solve the group, sparse-group and overlapping sparse-group Lasso models respectively. Sections 3, 7 and 9 of \cite{SLEP} provide details for these first-order proximal gradient methods. FISTA, without backtracking, was applied as described in \cite{BT09}. All iterative sequences in the algorithms were initialized at zero. For comparison, all of the Lasso models were tested with the same data matrix $\bA$ and response vectors $\bb^i$. Figure \ref{fig:comp_results} displays the convergence results for C-FISTA, SLEP, ADMM and FISTA on the three Lasso models. In Figure \ref{fig:comp_results}, the y-axis measures the absolute difference from the optimal objective value at each iteration of the individual algorithms. We took the final value of the objective function of C-FISTA as the exact optimal solution which agreed, as can be seen in Figure \ref{fig:comp_results}, to within machine tolerance of the final iterates produced by each of the methods. Note, we fine-tuned the augmented Lagrangian parameter $\rho$ in the ADMM algorithm to achieve the optimal performance for the method.

The computational results displayed in Figure \ref{fig:comp_results} clearly demonstrate the proven global linear convergence of C-FISTA. While ADMM, SLEP and FISTA all display linear convergence properties in some of the tests, C-FISTA significantly outperformed the other methods. Furthermore, C-FISTA maintained a robustness between the various Lasso models. While SLEP's convergence rate suffered as the subvector group sizes increased in the group Lasso model, the convergence rate of C-FISTA did not suffer in the group or sparse-group Lasso models as the group sizes increased. Overall, the tests demonstrate C-FISTA's effectiveness and robustness to varying subvector group sizes throughout the various Lasso formulations. 
%
%The results for the overlapping sparse-group Lasso tests are of particular interest because SLEP is considered the state-of-the-art first order approach for this model, and C-FISTA was about two times quicker than SLEP in our tests.
%
% (CUT FROM Paragraph above) -  {\it sgLogisticR.m} to solve the group, sparse-group and overlapping sparse-group Lasso models and sparse-group logistic regression model respectively.
%===================================
\begin{figure}%\label{fig:comp_results}
\centering
\begin{minipage}[c]{0.32\textwidth}
\centering 
% Tests labels at the top of the figure. 
{\bf \underline{GL Tests}}
\end{minipage}%\hspace{0.02in}
\begin{minipage}[c]{0.32\textwidth}
\centering 
{\bf \underline{SGL Tests}}
\end{minipage}%\hspace{0.02in}
\begin{minipage}[c]{0.32\textwidth}
\centering 
{\bf \underline{OSGL Tests}}
\end{minipage}
\vspace{0.05in}

% Figures from the experiment
\begin{minipage}[c]{0.32\textwidth}
\label{fig:GL_10}
\centering 
\includegraphics[width = \textwidth]{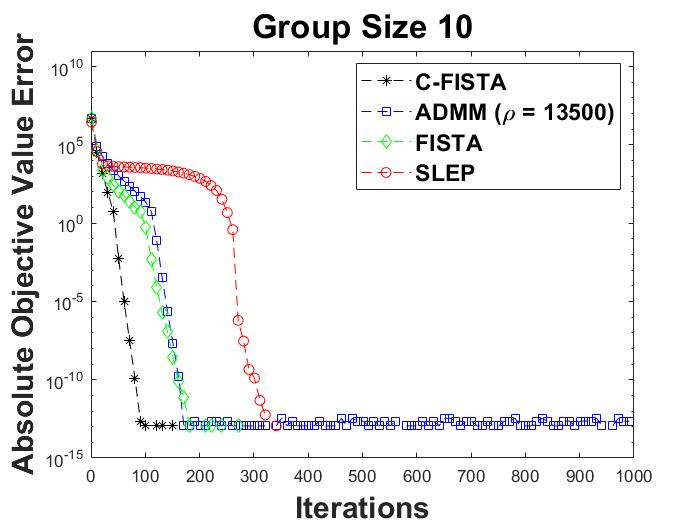}
\end{minipage}\hspace{0.02in}
\begin{minipage}[c]{0.32\textwidth}
\label{fig:GL_100}
\centering 
\includegraphics[width = \textwidth]{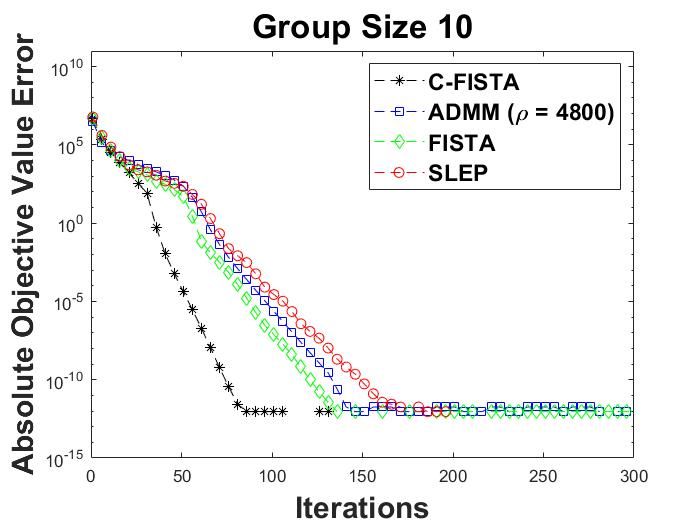}
\end{minipage}\hspace{0.02in}
\begin{minipage}[c]{0.32\textwidth}
\label{fig:GL_200}
\centering 
\includegraphics[width =\textwidth]{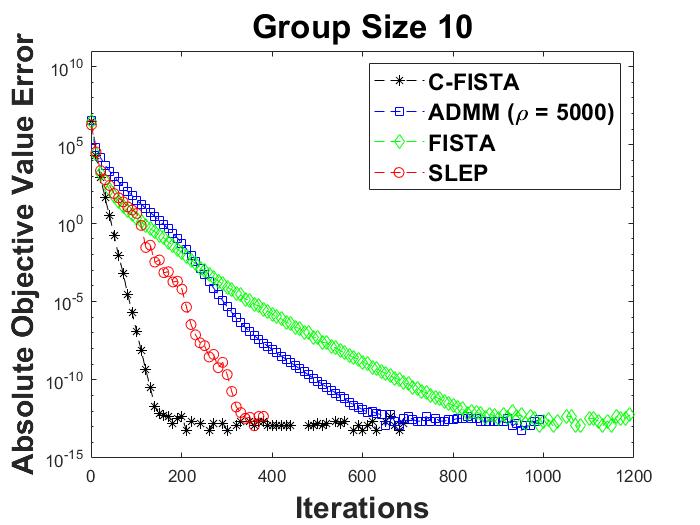}
\end{minipage}

\vspace{0.025in}
\begin{minipage}[c]{0.325\textwidth}
\label{fig:SGL_10}
\centering 
\includegraphics[width = \textwidth]{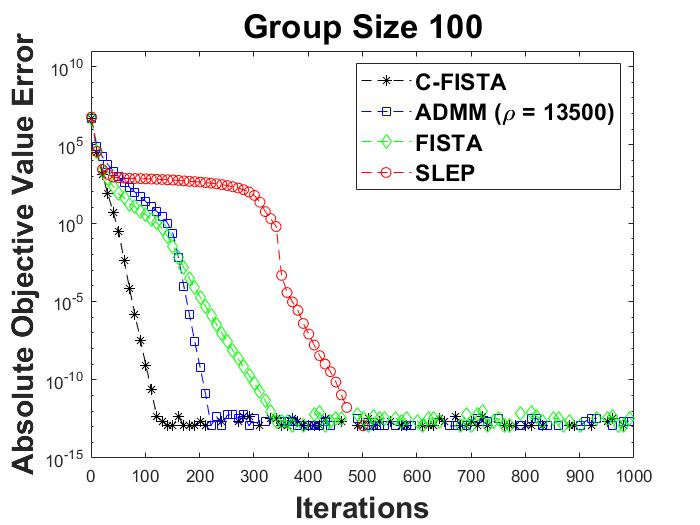}
\end{minipage}\hspace{0.02in}
\begin{minipage}[c]{0.325\textwidth}
\label{fig:SGL_100}
\centering 
\includegraphics[width = \textwidth]{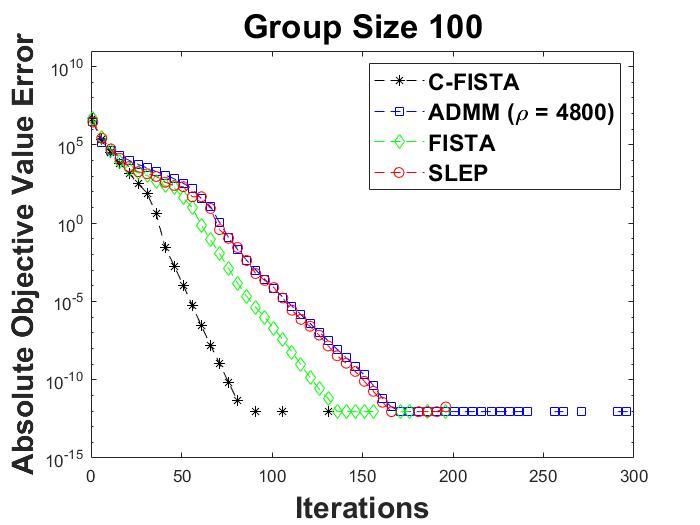}
\end{minipage}\hspace{0.02in}
\begin{minipage}[c]{0.325\textwidth}
\label{fig:SGL_200}
\centering 
\includegraphics[width =\textwidth]{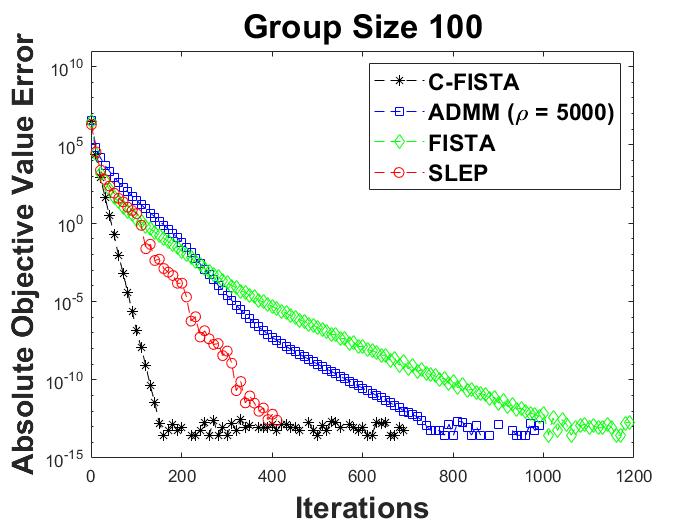}
\end{minipage}

\vspace{0.025in}
\begin{minipage}[c]{0.325\textwidth}
\label{fig:OSGL_10}
\centering 
\includegraphics[width = \textwidth]{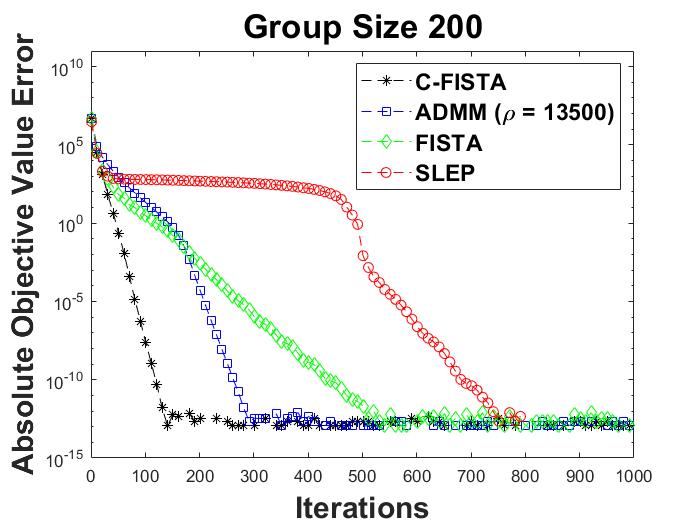}
\end{minipage}\hspace{0.02in}
\begin{minipage}[c]{0.325\textwidth}
\label{fig:OSGL_100}
\centering 
\includegraphics[width = \textwidth]{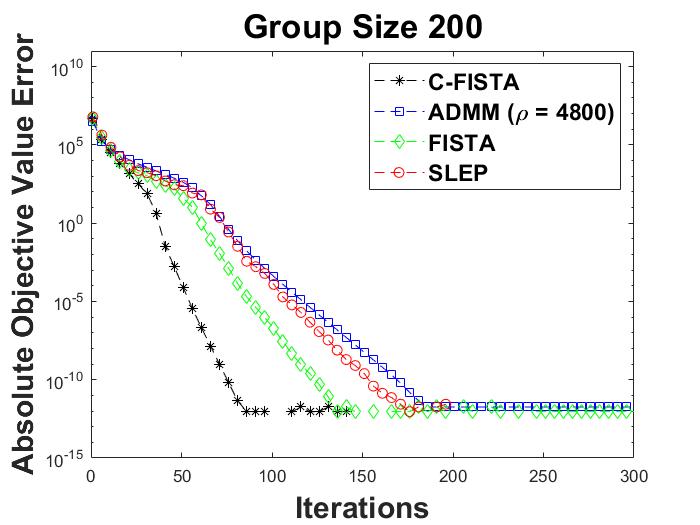}
\end{minipage}\hspace{0.02in}
\begin{minipage}[c]{0.325\textwidth}
\label{fig:OSGL_200}
\centering 
\includegraphics[width =\textwidth]{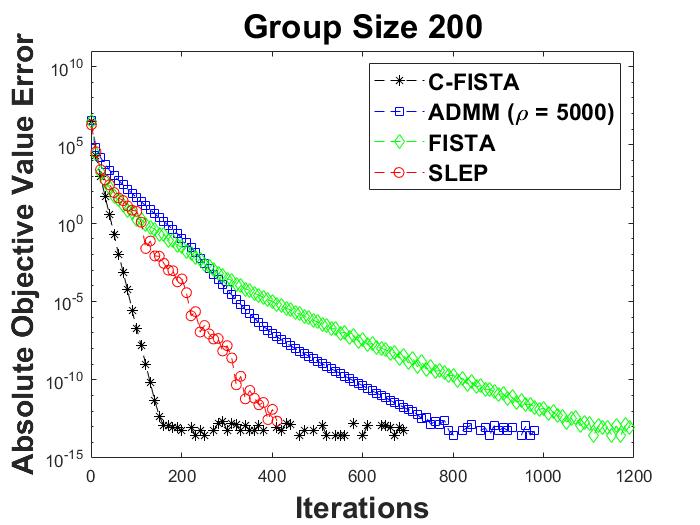}
\end{minipage}
\caption{The numerical results for the Lasso models $(GL)$, $(SGL)$ and $(OSGL)$ with three different subvector groupings. The $y$-axis is given in logarithmic scale and measures the absolute difference in the objective value and the optimal objective value at each iteration; $\gamma = 5$, $\gamma_1 = \gamma_2 = 10$ and $\gamma_1 = \gamma_2 = 0.1$ for the $(GL)$, $(SGL)$ and $(OSGL)$ models respectively}\label{fig:comp_results}
\end{figure}

%============================================================================
\subsubsection{Underdetermined Group Lasso Numerical Tests}\label{undet_GL}
%============================================================================
The second set of Lasso experiments focused on underdetermined data. In particular, the group Lasso model was studied with a rank deficient data matrix. With underdetermined data, the assumptions for Theorem \ref{thm:iter-complexity} are not fully realized; however, as previously claimed and will be shown, in practice asymptotic linear convergence is still possible when the global assumptions (A2) and/or (A3) are not fully realized. To demonstrate this we conducted an extensive numerical experiment comparing FISTA, without backtracking, to C-FISTA on the group Lasso model with underdetermined data. 

For the experiment, we let $H(\bx) = \frac{1}{2}\|\bx - \bb\|^2$ and $\bB(\bx) = \bA\bx$. Therefore, the strong convexity and gradient Lipschitz constants for $H$ are $\mu = L = 1$,  $\xi = 0$, and $r = \lambda_{\text{max}}(\bA^\top \bA)$. The only remaining constant to compute is $\tau > 0$ which by (A2) needs to satisfy, 
%The only remaining constant necessary to apply C-FISTA is $\tau > 0$ which by (A2) needs to satisfy, 
\[ 
\tau\|\bx-\by\|^2 \leq \| \bA(\bx-\by)\|^2, 
\]
for all $\bx, \by \in \RR^n$. Since $\bA$ is not of full-column rank, this assumption cannot hold; however, choosing a small positive value for $\tau$ ensures the inequality will hold for a sufficient number of iterates generated by Algorithm \ref{alg:C FISTA alg}. This enables C-FISTA to maintain its convergence though assumption (A2) is not formally satisfied. 

In our underdetermined group Lasso experiments, we generated synthetic data matrices and response vectors as in Section \ref{lasso_test_1}. Randomly generated matrices with dimensions: $2500 \times 5000, 1200 \times 5000$ and $600 \times 5000$ were constructed, and each of these group Lasso problem sizes were solved under the three subvector groupings from Section \ref{lasso_test_1}. We generated ten instances of these nine group Lasso settings and solved the models with FISTA and C-FISTA under three different parameter settings for $\tau$. Table \ref{table_1} displays the results for these ninety numerical experiments. Each entry in the table provides the average number of iterations for the respective algorithms to converge to within $10^{-10}$ of the optimal objective value.

From Table \ref{table_1} we observe C-FISTA is convergent for a range of $\tau$ values demonstrating a robustness to miss-specification. Of the different values for $\tau$ selected, only the selection of $\tau = 1$ failed to converge after 20,000 iterations and this was only for the case $\bA \in \RR^{600 \times 5000}$ and the group variable size was 200. Second, we note C-FISTA outperformed FISTA consistently with the best algorithm for each setting being C-FISTA with either $\tau = 0.1$ or $\tau = 1$. A visual representation of the performance comparison is given in Figure \ref{fig:GL_comp_results}. From the convergence plots we see C-FISTA would often display asymptotic linear convergence when the group sizes were 100 and 200 while FISTA would converge sub-linearly. When the group size was 10, irrespective of the dimension of the data matrix, C-FISTA and FISTA were comparable and demonstrated asymptotic linear convergence; however, with $\tau=1$, C-FISTA converged between 300 and 500 iterations sooner on average then FISTA  in this setting. These experiments demonstrate that C-FISTA maintains linear convergence properties even when some of the stated assumptions are not strictly met and showcases the robustness of the parameter $\tau$ to miss-specification. 

\begin{table}[tbhp]
\centering
\caption{
This table shows the average number of iterations for C-FISTA and FISTA to converge to within $10^{-10}$ of the optimal objective value for 10 randomized Group Lasso problems of varying dimensions and group vector sizes. An entry of {\it NaN} means the given algorithm did not converge to within the stated tolerance after 20,000 iterations for at least one of the ten randomized tests} \label{table_1}
\hspace{-0.05in}
\begin{tabular}{|c|lllllllll|} 
\hline 
\textbf{} & \multicolumn{9}{c|}{\textbf{Data Dimension}} \\ \cline{2-10} 
\multicolumn{1}{|l|}{\textbf{Method}} & \multicolumn{3}{c|}{2500 x 5000} & \multicolumn{3}{c|}{1200 x 5000} & \multicolumn{3}{c|}{600 x 5000} \\ \cline{2-10} 
\multicolumn{1}{|l|}{} & \multicolumn{3}{c|}{\textbf{Group Size}} & \multicolumn{3}{c|}{\textbf{Group Size}} & \multicolumn{3}{c|}{\textbf{Group Size}} \\ \hline 
C-FISTA & \multicolumn{1}{c}{10} & 100 & \multicolumn{1}{l|}{200} & \multicolumn{1}{c}{10} & 100 & \multicolumn{1}{l|}{200} & \multicolumn{1}{c}{10} & 100 & \multicolumn{1}{c|}{200} \\ \hline 
\multicolumn{1}{|r|}{$\tau = 0.01$}   &  2194   &  4069   & \multicolumn{1}{l|}{15718}     & 2170    & 13546    & \multicolumn{1}{l|}{13288}    & 2208   & 12014    & 11565  \\ 
\multicolumn{1}{|r|}{$\tau = 0.1$} &  1129   &  2202   & \multicolumn{1}{l|}{\bf5331}    & 1202    & {\bf 4686}    & \multicolumn{1}{l|}{\bf4436}    & 1439   & {\bf 4088}    & {\bf4140}   \\  
\multicolumn{1}{|r|}{$\tau = 1$} &   {\bf 624}   &  {\bf 2049}   & \multicolumn{1}{l|}{15517}    & {\bf 833}    & 14096    & \multicolumn{1}{l|}{15473}    & {\bf 1292}   & 17327    &  NaN   \\ \hline 
FISTA                              &   945   &  3200   & \multicolumn{1}{l|}{ NaN}    & 1284    &  NaN   & \multicolumn{1}{l|}{ NaN}    & 1780   &  NaN    &  NaN   \\ \hline 
\end{tabular} 
\vspace{0.01in}
\end{table}

\begin{figure}[h]%\label{fig:GL_comp_results}
\centering
\begin{minipage}[c]{0.32\textwidth}
\centering 
% Tests labels at the top of the figure. 
{\bf \underline{Dimension: $2500 \times 5000$}}
\end{minipage}\hspace{0.02in}
\begin{minipage}[c]{0.32\textwidth}
\centering 
{\bf \underline{Dimension: $1200 \times 5000$}}
\end{minipage}\hspace{0.02in}
\begin{minipage}[c]{0.32\textwidth}
\centering 
{\bf \underline{Dimension: $600 \times 5000$}}
\end{minipage}
\vspace{0.05in}

% Figures from the experiment
\begin{minipage}[c]{0.325\textwidth}
\centering 
\includegraphics[width = \textwidth]{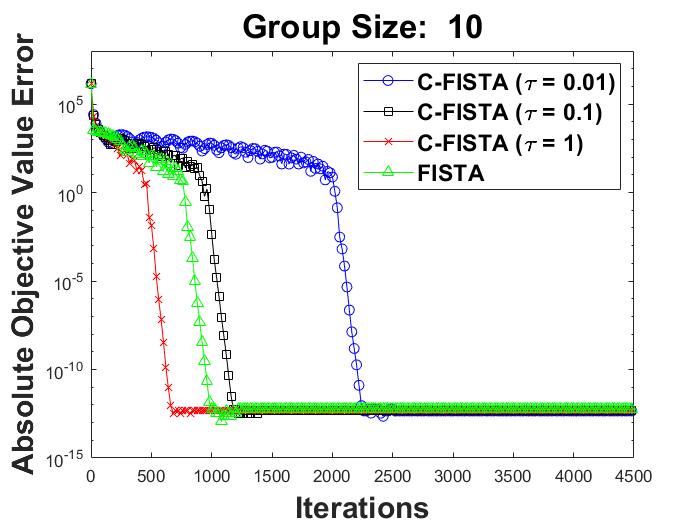}
\end{minipage}\hspace{0.02in}
\begin{minipage}[c]{0.325\textwidth}
\centering 
\includegraphics[width = \textwidth]{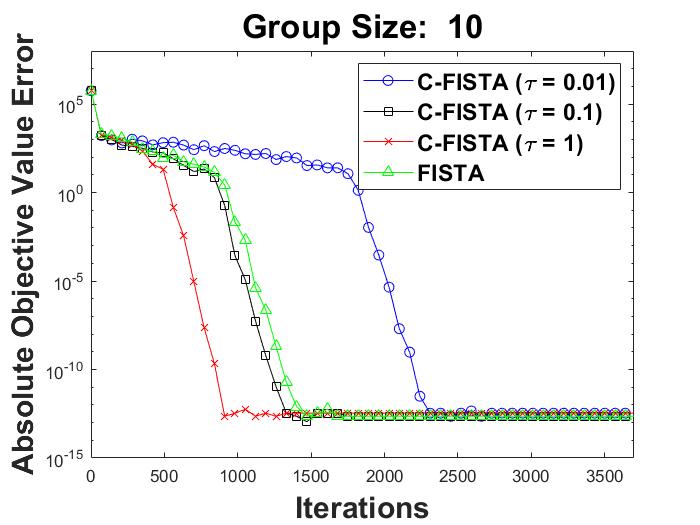}
\end{minipage}\hspace{0.02in}
\begin{minipage}[c]{0.325\textwidth}
\centering 
\includegraphics[width =\textwidth]{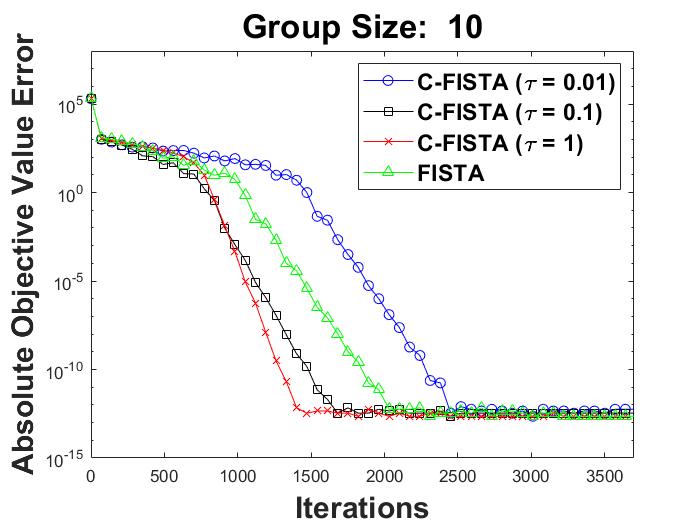}
\end{minipage}

\vspace{0.025in}
\begin{minipage}[c]{0.325\textwidth}
\centering 
\includegraphics[width = \textwidth]{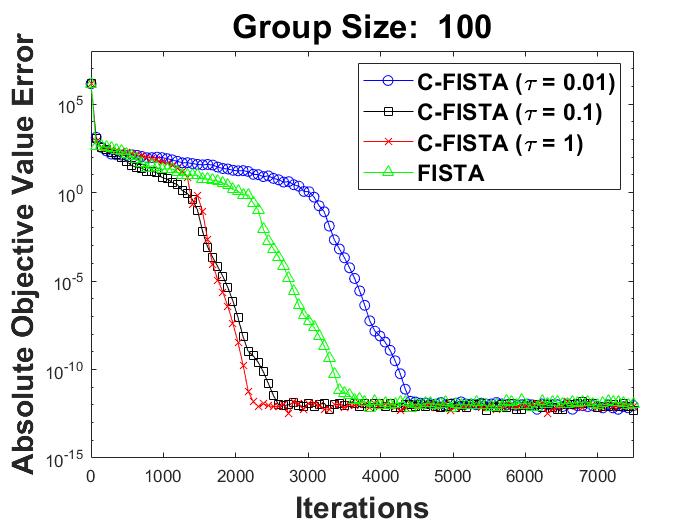}
\end{minipage}\hspace{0.02in}
\begin{minipage}[c]{0.325\textwidth}
\centering 
\includegraphics[width = \textwidth]{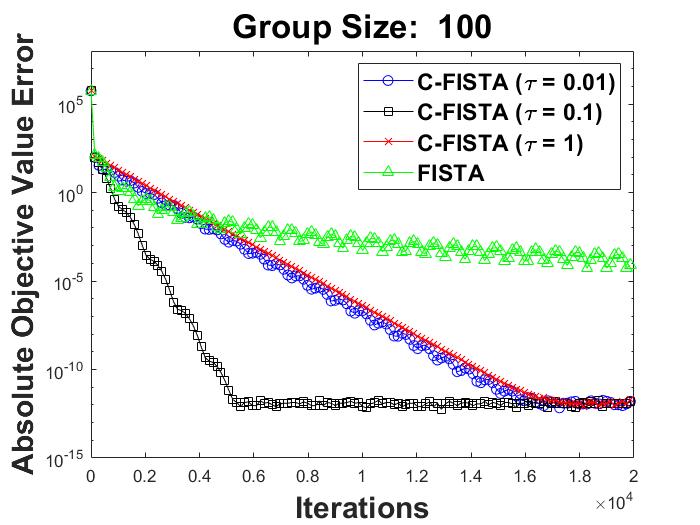}
\end{minipage}\hspace{0.02in}
\begin{minipage}[c]{0.325\textwidth}
\centering 
\includegraphics[width =\textwidth]{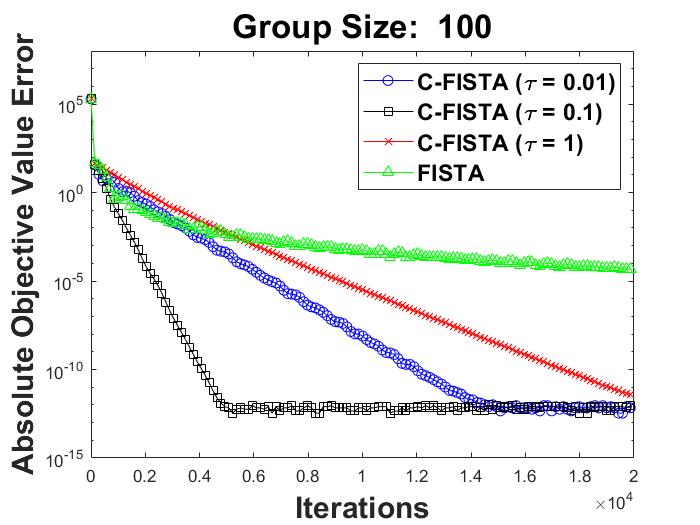}
\end{minipage}

\vspace{0.025in}
\begin{minipage}[c]{0.325\textwidth}
\centering 
\includegraphics[width = \textwidth]{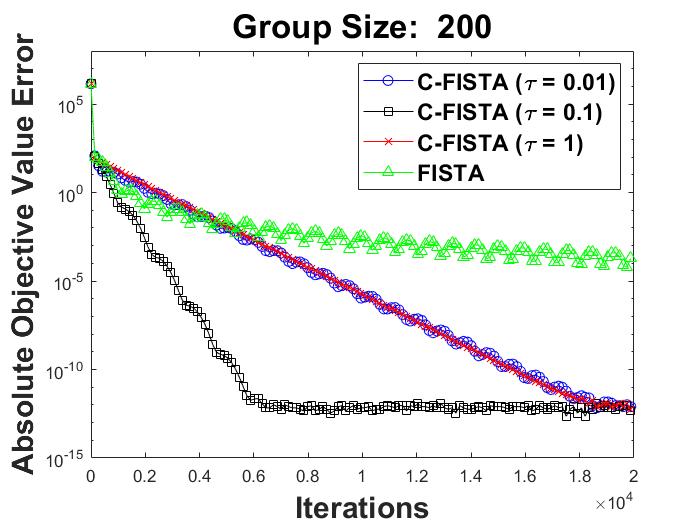}
\end{minipage}\hspace{0.02in}
\begin{minipage}[c]{0.325\textwidth}
\centering 
\includegraphics[width = \textwidth]{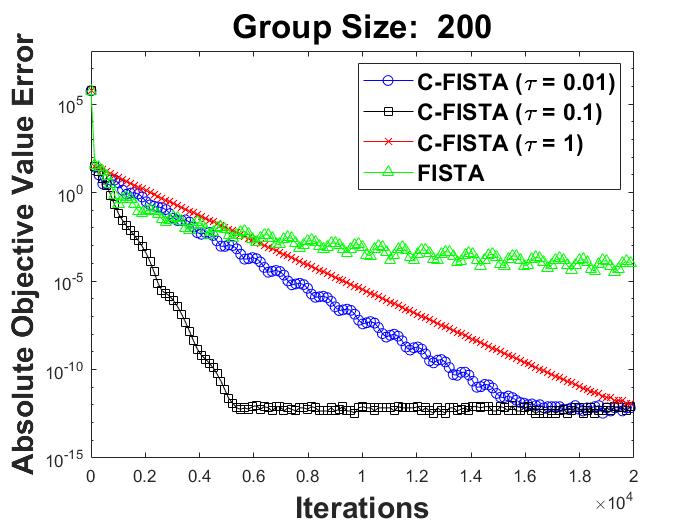}
\end{minipage}\hspace{0.02in}
\begin{minipage}[c]{0.325\textwidth}
\centering 
\includegraphics[width =\textwidth]{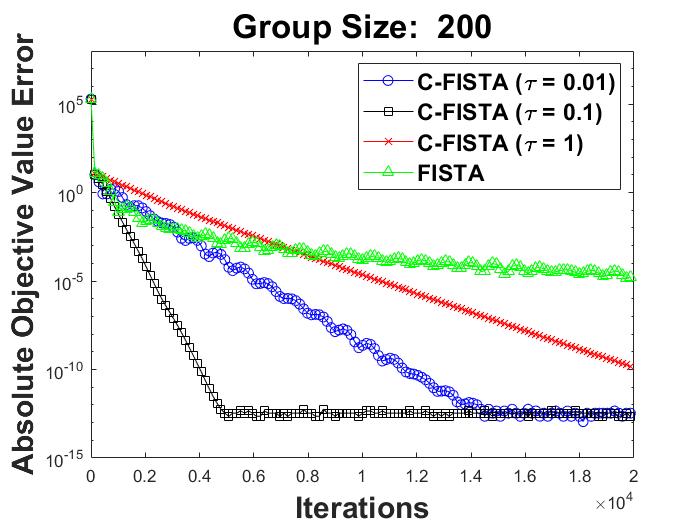}
\end{minipage}

\caption{Sample plots of the underdetermined group Lasso tests between C-FISTA and FISTA with varying problem dimensions and subvector group sizes. C-FISTA was implemented with three different parameter settings for $\tau$ in each test. The penalty parameter was set as $\gamma = 5$ for all of the tests}\label{fig:GL_comp_results}
\end{figure}

%============================================================================
\subsection{C-FISTA for Sparse-Group Logistic Regression}\label{SGLR}
%============================================================================

In this section we describe how to apply {C-FISTA} to solve the sparse-group logistic regression model,
\[
\begin{array}{lll}
 &\min &\frac{1}{m} \sum_{i=1}^{m} \ln\left( 1 +\exp\left( -y_i \left( \ba_i^\top \bx + b \right) \right) \right) + \gamma_1 \sum_{j \in \GG} \| \bx(j) \| + \gamma_2 \| \bx\|_1 \\
& \mbox{s.t.} & \bx \in \RR^n, b \in \RR.
\end{array}
\]
{As in the Lasso formulations, a decomposition must be chosen to tackle the regression model with C-FISTA. For continuity with the previous discussion, we will utilize the decomposition $\bB(\bx) = \bx$ and $H(\bx) =\frac{1}{m} \sum_{i=1}^{m} \ln( 1 +\exp( -y_i \langle (\ba_i^\top, 1)^\top, \bx \rangle)).$ Under these definitions for $H$ and $\bB$,} the sparse-group Lasso and sparse-group logistic regression models only differ with respect to their strongly convex term, so the procedure for solving the sparse-group logistic regression model with {C-FISTA} is very similar to Algorithm \ref{alg:C_FISTA_SGL}. Only two alterations to Algorithm \ref{alg:C_FISTA_SGL} are required to solve the sparse-group logistic regression model. One, we must update $\bd$ in Algorithm \ref{alg:C_FISTA_SGL} to be,
\begin{equation}\label{eqn:SGLR_grad}
 \bd = \by^k - \frac{1}{L} \nabla H(\by^k) = \by^k + \frac{1}{mL}\sum_{i=1}^{m} \frac{y_i \exp(-y_i (\ba_i^\top \by^k + b))}{1 + \exp(-y_i( \ba_i^\top \by^k +b ))} \begin{pmatrix} \ba_i \\ 1 \end{pmatrix},
\end{equation}
and second we require new estimates for the bounds of the Lipschitz and strong convexity constants $L$ and $\mu$. Updating $\bd$ is straightforward, but obtaining tight upper and lower bounds on the Lipschitz and strong convexity constants can become a challenging task as mentioned previously. Algorithm \ref{alg:C_FISTA_SGLR} presents {C-FISTA} for solving the sparse-group logistic regression model. In our implementation of Algorithm \ref{alg:C_FISTA_SGLR}, we fixed the Lipschitz and strong convexity parameters though backtracking procedures could be utilized. 

\begin{algorithm}
\caption{C-FISTA for $(SGLR)$}
\begin{algorithmic}\label{alg:C_FISTA_SGLR}
{\STATE \textbf{Input:} Constants $\mu$ and $L$; penalty parameters $\gamma_1, \gamma_2 > 0$; vector partition $\mathcal{G}$.}
\STATE \textbf{Step 0.} Choose any $(\bx^0, b^0, \bz^0) \in \mathbb{R}^n \times \mathbb{R} \times \RR^{n+1}$. Let $k:=0$.  

\STATE \textbf{Step 1.} Let
\[ \by^k := \frac{1}{1 + \theta}\begin{pmatrix} \bx^k  \\ b^k \end{pmatrix}+ \frac{\theta}{1 + \theta} \bz^k. \]

\STATE \textbf{Step 2.} {For $j\in \mathcal{G}$}: define $\mathbf{\Omega}(j) := \|\d(j)+\textbf{Th}_{\gamma_2/L}(-\d(j))\|$ and,
\[
\bx(j)^{k+1}\hspace{-0.02in} :=\hspace{-0.03in} \left\{ \begin{array}{ll}
\hspace{-0.02in} 0, & \hspace{-0.05in}\mbox{ if $\mathbf{\Omega}(j) \le \gamma_1/L$}; \\
\hspace{-0.07in}\left(\frac{\|d(j)+\Th_{\gamma_2/L}(-d(j))\|- \gamma_1/L)}{\|d(j)+\Th_{\gamma_2/L}(-d(j))\|}\right) \left(d(j)+\Th_{\gamma_2/L}(-d(j))\right), &\hspace{-0.05in} \mbox{ if $\mathbf{\Omega}(j) > \gamma_1/L$};
\end{array} \right.
\]
where $d$ is as defined in \eqref{eqn:SGLR_grad} and $d(j)$ is the subvector of $d$ corresponding to the subvector $x(j)$.

\STATE \textbf{Step 3.} \[ b^{k+1} = y^{k+1}_{n+1} + \frac{1}{mL} \sum_{i=1}^{m} \frac{y_i \exp(-y_i (a_i^\top y^k + b))}{1 + \exp(-y_i( a_i^\top y^k +b ))}.\]

\STATE \textbf{Step 4.}  \[ z^{k+1}:= (1-\theta) z^k + \theta y^k + {\alpha} \left[ \begin{pmatrix} x^{k+1} \\ b^{k+1} \end{pmatrix} - y^k \right],\]
{where $\theta = \sqrt{\mu/L}$ and $\alpha = \sqrt{L/\mu}$.}\vspace{3mm}
\STATE \textbf{Step 5.} Let $k:= k+1$; return to Step 1 until convergence.
\end{algorithmic}
\end{algorithm}

Note, computing $b^{k+1}$ at each iteration is simple relative to $\bx^{k+1}$ because the subprobem \eqref{eqn:x-def} is decomposable into a minimization over $\bx$ and $b$ respectively. Since the minimization over $b$ contains no regularization terms, the first-order optimality conditions provide a simple update for $b$.

The astute observer will note $H$ is strictly instead of strongly convex making $\mu = 0$ and leaving the necessary assumptions unsatisfied. Though this is technically true, C-FISTA is undeterred. As in the underdetermined group Lasso tests, this decomposition showcases the robustness of C-FISTA when the assumptions are not formally met. It is clear the sparse-group logistic regression model could be equivalently expressed as a constrained problem; therefore, by rewriting the model as an equivalent constrained problem, $H$ would become strongly convex on the problem domain; however, as is seen in Figure \ref{fig:SGLR_comp_results}, C-FISTA outperforms both SLEP and FISTA without solving an equivalent model where the assumptions are formally satisfied. 

\vspace{-0.3in}
 
%============================================================================
\subsection{Sparse-Group Logistic Regression Numerical Experiments}\label{SGLR_Exp}
%============================================================================
In this section, we describe the results of the numerical experiments on the sparse-group logistic regression model. The set-up for the numerical experiments on the sparse-group logistic regression model were similar to the set-up for the Lasso models in \ref{Lasso_Exp}. We generated the data matrix $\bA\in \RR^{500 \times 5000}$ from the standard normal distribution letting $\ba_i$ for $i=1,2,\hdots, 500$ be the rows of $\bA$, and set $\by = (-\mathbf{e}^\top, \mathbf{e}^\top)^\top$ where $\mathbf{e} \in \RR^{250}$ is the vector of all ones. We conducted three experiments from the generated data using the same subvector groups, $\mathcal{G}_1$, $\mathcal{G}_2$ and $\mathcal{G}_3$, utilized in the group and sparse-group Lasso tests. We set $\gamma_1 = \gamma_2 = 0.01$ in our tests to ensure a non-trivial solution with substantial sparsity.

In our experiments, we compared C-FISTA, SLEP and FISTA. With regards to SLEP, we utilized {\it sgLogisticR.m}, which is a first-order proximal gradient method, to solve the model. The computational results displayed in Figure \ref{fig:SGLR_comp_results} clearly demonstrate the proven global linear convergence of C-FISTA. In the logistic regression tests, C-FISTA was a vast improvement over both FISTA and SLEP converging well over a 1000 iterations sooner to within machine precision.

%%% INSERT THE FIGURES HERE%%%
\begin{figure}%\label{fig:SGLR_comp_results}
\centering
\begin{minipage}[c]{0.325\textwidth}
\label{fig:SGLR_10}
\centering 
\includegraphics[width = \textwidth]{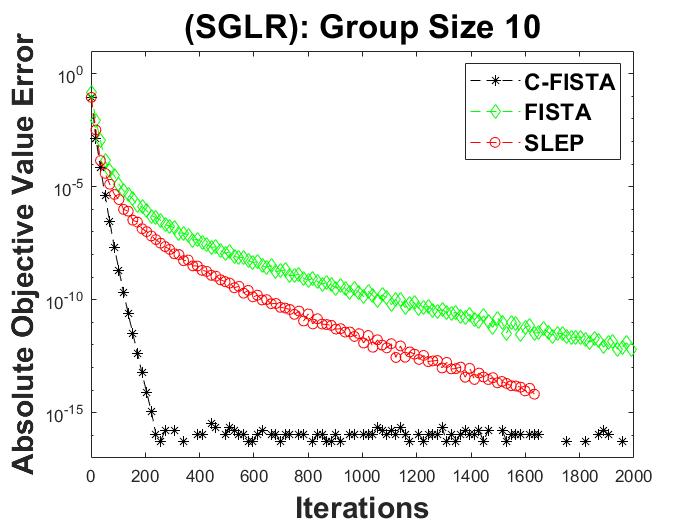}
\end{minipage}\hspace{0.02in}
\begin{minipage}[c]{0.325\textwidth}
\label{fig:SGLR_100}
\centering 
\includegraphics[width = \textwidth]{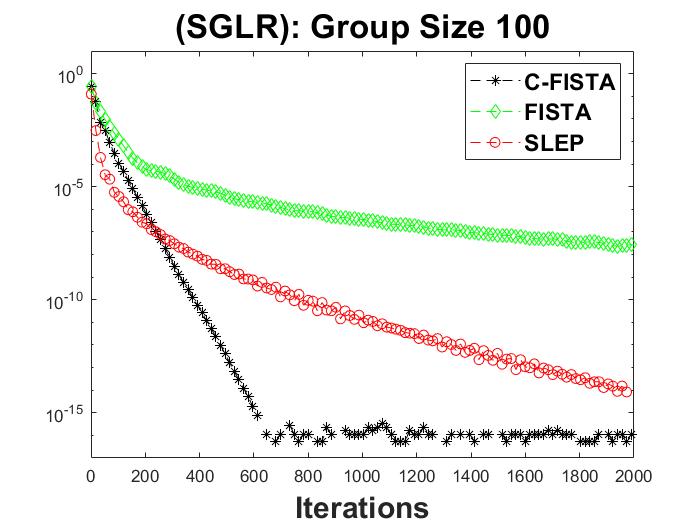}
\end{minipage}\hspace{0.02in}
\begin{minipage}[c]{0.325\textwidth}
\label{fig:SGLR_200}
\centering 
\includegraphics[width =\textwidth]{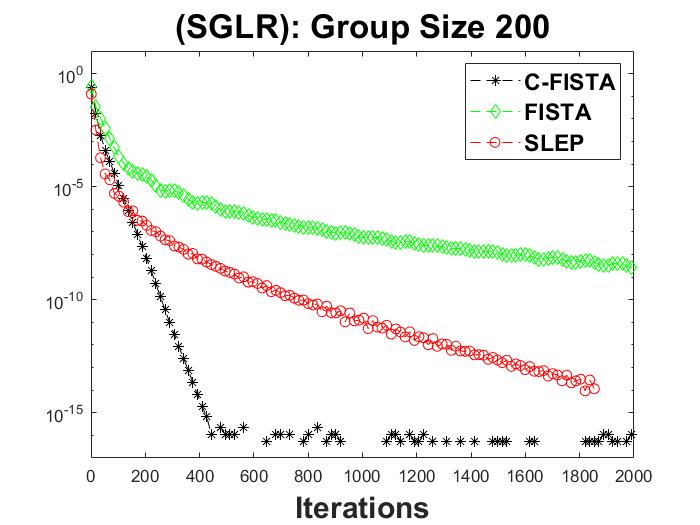}
\end{minipage}

\caption{The numerical results for the sparse-group logistic regression model $(SGLR)$ with three different subvector groupings. The $y$-axis is given in logarithmic scale and measures the absolute difference in the objective value and the optimal objective value at each iteration; $\gamma_1 = \gamma_2$ were set to be $0.01$   }\label{fig:SGLR_comp_results}
\end{figure}

%============================================================================
\subsection{C-FISTA for Geometric Programming}\label{GeoProg}
%============================================================================
In this section we discuss how to solve an instance of the geometric programming model described in Section \ref{Lasso and log reg models}. In particular, we focus on the group regularized geometric program, 
\begin{equation}\label{eqn:geo_prog_ex}
\begin{array}{lll}
& \min & \;\; \sum_{k=1}^{K} c_k x_1^{a_{k1}} x_2^{a_{k2}}\hdots x_n^{a_{kn}} + \gamma \sum_{j \in \mathcal{G}} \|\bx(j)\|_2,  \\
&\mbox{s.t}& \bx \in \XX, 
\end{array}
\end{equation}
where $\XX = \{ \bx \in \RR^n \; | \; \|\bx - \bx_0 \|_2 \leq \Delta \} \subseteq \RR^n_{++}$ with $\Delta > 0$ and $\bx_0 \in \RR^n$,  $\gamma > 0 $, $c_k > 0 $ and $\mathbf{a_{k}} := (a_{k1}, \hdots, a_{kn})^\top \leq 0$ for $k=1, \hdots, K$, and $\mathcal{G}$ is a non-overlapping partition of the vector $\bx$. Letting, 
\[ H(\bx):= \sum_{k=1}^{K} c_k e^{\ba_k^\top \bx}, \;\; \bB(\bx) := \left( \ln(x_1), \hdots, \ln(x_n) \right)^\top, \;\;  \; R(\bx):= \gamma \sum_{j \in \mathcal{G}} \|\bx(j)\|_2, \]
we can rewrite \eqref{eqn:geo_prog_ex} in the format of our standard composite optimization model. From the characteristics of $\XX$ it follows $H$ is gradient Lipschitz on the problem domain. Additionally, $H$ will be strongly convex over the domain provided  the vectors $\{a_1, \hdots, a_K\}$ span $\RR^n$. Note, this is a fairly weak condition if $K > n$. Additionally, similar to our underdetermined group Lasso discussion, although constants $r$ and $\tau$ do not exist to satisfy the necessary constraints over the span of the constraint set, such constants exist which satisfy the necessary inequalities over the constraint set. By Taylor's theorem and simple bounding we can see, 
\[ \tau \|\bx - \by\|^2 \leq \|\bB(\bx) - \bB(\by)\|^2 \leq r \|\bx-\by\|^2, \;\; \forall \bx, \by \in \XX,\]
with $r = 1/\min_{1 \leq i \leq n} \left[(\bx_0)_i - \Delta \right]^2$ and $\tau = 1/\max_{1 \leq i \leq n} \left[(\bx_0)_i + \Delta \right]^2$. Thus, with these values for $r$ and $\tau$, along with appropriate estimates for the other parameters $\mu, L$ and $\xi$, C-FISTA can be employed to solve instances of this subclass of geometric programs. The only difference in implementing C-FISTA for this problem class is a constrained optimization model of the form given in \eqref{eqn:GL_subprob} must be solved to compute the proximal mapping. In our numerical experiments, we apply a simple projected gradient method to solve this subproblem to within a specified tolerance. We now present the numerical results from the geometric programming model.

%============================================================================
\subsection{Geometric Programming Numerical Experiments}\label{GeoProg_exp}
%============================================================================
To test the performance of C-FISTA on \eqref{eqn:geo_prog_ex}, we generated random instances of the geometric programming model and compared Algorithm \ref{alg:C FISTA alg} to FISTA. For our randomly generated examples we let: 
\[\mathcal{G} = \{ (1,\hdots,100), (101, \hdots ,200), \hdots, (901, \hdots ,1000)\},\]
$K = 5000$, $\bx_0 = 10 \mathbf{e}$, $\Delta = 9$ and $\gamma = 0.1$. The values for the $c_k$'s and $\mathbf{a_k}$'s were randomly computed using the \textit{randn} command in MATLAB 2021b, and then they were transformed to ensure they held the proper sign. The $\mathbf{a_k}$'s were additionally normalized so that the elements in each vector $\mathbf{a_k}$ summed to negative one. This normalization was done to avoid extremely ill-conditioned problems. 

Figure \ref{fig:GP_results} displays an example result from our testing. For this problem, C-FISTA and FISTA would both converge exceptional fast often within 10-30 iterations. In general, FISTA would slightly outperform C-FISTA which might be a result of the parameter estimation. Nevertheless, C-FISTA demonstrated rapid convergence in a setting where the formal assumptions were not completely satisfied and a non-linear decomposition of the objective function was utilized.   

\begin{figure}%\label{fig:GP_results}
\centering
\begin{minipage}[c]{\textwidth}
\centering 
\includegraphics[width = 0.5\textwidth]{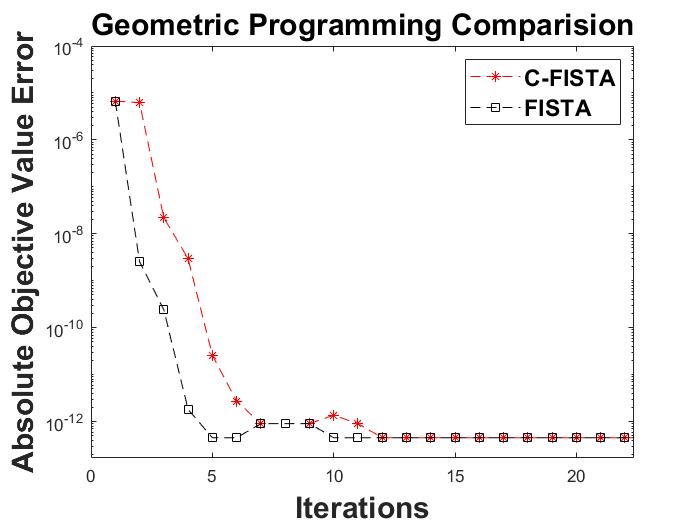}
\end{minipage}
\caption{Numerical results from the geometric programming tests. C-FISTA and FISTA were both implemented with the same Lipschitz constant estimate and both demonstrated rapid convergence to the optimal solution. The $y$-axis is given in logarithmic scale and measures the absolute difference in the objective value and the optimal objective value at each iteration }\label{fig:GP_results}
\end{figure}
}
%============================================================================
\section{Conclusions}\label{conclusions}
%============================================================================

In this paper we developed an accelerated composite version of FISTA, C-FISTA, which handles the composite optimization model \eqref{eqn:genCCO} and is a generalization of GFISTA \cite{CC19,CP16}. We proved global linear convergence for C-FISTA without having a strongly convex objective function, and demonstrated through Fenchel duality the breadth of convex models which are solvable via C-FISTA in Section \ref{Fenchel Duality}. In Section \ref{numerical results}, we demonstrated through several numerical experiments C-FISTA was able to obtain linear convergence even in settings where the formal assumptions were not satisfied. Furthermore, C-FISTA outperformed ADMM, the software package SLEP and the seminal FISTA algorithm in both Lasso and logistic regression models.

The following lines of directions could be of interest as future research topics. First, the conditions underlying our global linear convergence results may be weakened. Second, one may further study possible {\it adaptive}\/ schemes to implement C-FISTA without requiring the exact knowledge of the problem parameters. Finally, it is interesting to study how the new algorithm behaves beyond the scope of convex optimization.

%============================================================
% Declarations
%============================================================
\section{Statements and Declarations}

\noindent {\bf Funding.} This material is based upon work supported by the National Science Foundation Graduate Research Fellowship Program under Grant No. 1839286. Any opinions, findings, and conclusions or recommendations expressed in this material are those of the author(s) and do not necessarily reflect the views of the National Science Foundation.\\

\noindent{\bf Competing Interests.} The authors have no relevant financial or non-financial interests to disclose. \\

\noindent {\bf Data Availability.} The datasets generated during and/or analysed during the current study are available from the corresponding author on reasonable request.

%============================================================
% References
%============================================================
%\bibliographystyle{spmpsci}
%\bibliography{references}

%============================================================
\end{document}